\documentclass{siamonline250211}
\usepackage{a4,latexsym,exscale,theorem,epsfig}
\usepackage{amssymb,psfrag,epsf,amsmath,verbatim,bbm,float}
\usepackage{pifont}
\usepackage{mathtools}
\usepackage{listings}
\usepackage{color} 
\usepackage{mathbbol}
\usepackage{enumerate}
\newtheorem{remark}[theorem]{Remark}

\usepackage{geometry}
\usepackage{accents}
\usepackage{algorithm}
\usepackage{algpseudocodex}
\counterwithout{algorithm}{section}

\usepackage{todonotes}
\usepackage{caption}
\usepackage{amsmath,scalerel,graphicx}
\newcommand{\medoplusscale}{1.03}
\newcommand{\medoplusraise}{0.02ex}
\newcommand{\medoplus}{%
  \mathop{\mathpalette\medoplusaux\relax}\displaylimits
}
\newcommand{\medoplusaux}[2]{%
  \vcenter{\hbox{%
    \raisebox{\medoplusraise}{%
      \scalebox{\medoplusscale}{$#1\scalerel*{\bigoplus}{\sum}$}%
    }%
  }}%
} 
\newcommand{\bigtimes}{%
  \mathop{\scalerel*{\times}{\prod}}\displaylimits
}
 
\begin{document}
\newcommand {\eps} {\varepsilon}
\newcommand {\Z} {\mathbbm{Z}}
\newcommand {\R} {\mathbbm{R}}
\newcommand {\T} {\mathbbm{T}}
\newcommand {\Q} {\mathbbm{Q}}
\newcommand {\N} {\mathbbm{N}}
\newcommand {\C} {\mathbbm{C}}
\newcommand {\I} {\mathbbm{I}}
\newcommand {\dist} {{\rm{dist}}}
\newcommand {\cl}{\mathrm{cl}}
\newcommand {\PP} {\mathbbm{P}}
\newcommand {\ang} {\measuredangle}
\newcommand {\e} {{\rm{e}}}
\newcommand {\rank} {{\rm{rank}}}
\newcommand {\disc} {{\rm{disc}}}
\newcommand {\cont} {{\rm{cont}}}
\newcommand {\step} {{\rm{step}}}
\newcommand {\RK} {{\rm{RK}}}
\newcommand {\diff} {{\rm{diff}}}
\newcommand {\Span} {{\mathrm{span}}}
\newcommand {\Skew} {{\mathrm{skew}}}
\newcommand {\card} {{\rm{card}}}
\newcommand {\ED} {\mathrm{ED}}
\newcommand {\cA} {\mathcal{A}}
\newcommand {\cO} {\mathcal{O}}
\newcommand {\cF} {\mathcal{F}}
\newcommand {\cC} {\mathcal{C}}
\newcommand {\cN} {\mathcal{N}}
\newcommand {\cV} {\mathcal{V}}
\newcommand {\cG} {\mathcal{G}}
\newcommand {\cB} {\mathcal{B}}
\newcommand {\cD} {\mathcal{D}}
\newcommand {\cP} {\mathcal{P}}
\newcommand {\cQ} {\mathcal{Q}}
\newcommand {\cW} {\mathcal{W}}
\newcommand {\cT} {\mathcal{T}}
\newcommand {\cI} {\mathcal{I}}
\newcommand {\cL} {\mathcal{L}}
\newcommand {\bi} {\boldsymbol{i}}
\newcommand {\Sn}[1] {\mathcal{S}^{#1}}
\newcommand {\range} {\mathcal{R}}
\newcommand {\kernel} {\mathcal{N}}
\newcommand{\one}{\mathbb{1}}
\renewcommand{\thefootnote}{\fnsymbol{footnote}}
\newcommand{\rle}{\rotatebox[origin=c]{-90}{$\le$}}
\newcommand{\rl}{\rotatebox[origin=c]{-90}{$<$}}
\newcommand{\rg}{\rotatebox[origin=c]{-90}{$=$}}


\title{\bf Angular maps for visualizing rotational dynamics}

\author{Wolf-J\"urgen Beyn\footnotemark[1]\qquad
  Thorsten H\"uls\footnotemark[1]
}
\footnotetext[1]{Department of Mathematics, Bielefeld University,  
33501 Bielefeld, Germany \\
\texttt{beyn@math.uni-bielefeld.de}, \texttt{huels@math.uni-bielefeld.de}}

\maketitle


\begin{abstract}
We develop and analyze angular maps as a numerical tool for
visualizing and detecting rotational dynamics in nonlinear,
nonautonomous dynamical systems in discrete and continuous time. An
angular map assigns to each initial point the angular spectrum of the
variational equation along the corresponding trajectory. This spectrum
describes the long-time average rotation of subspaces transported by
the linearized dynamics, measured by maximal principal
angles. Building on the theory of angular spectra, we develop
efficient numerical algorithms based on forward and backward subspace
iteration.
We prove that these algorithms asymptotically provide
  angular spectral values  for subspaces that are dominant in either
  forward or backward time. 
Applications to Hénon maps, a planar flow, and the Lorenz
system illustrate how angular maps reveal rotational features across
phase space. For continuous-time systems, we prove that suitably
rescaled angular spectra of exact time-step maps converge to the
continuous angular spectrum in the Hausdorff metric as the step size
tends to zero. For autonomous systems, we also justify a simplified
algorithm that uses successive trajectory points to approximate the
angular range associated with the flow direction. 
\end{abstract}



\renewcommand*{\thefootnote}{\arabic{footnote}}
\section{Introduction}
\label{sec0}

The main goal of this paper is to quantify the rotational behavior of
dynamical systems in discrete and continuous time. We first consider
discrete-time systems
\begin{equation}\label{nonlin}
x_{n+1} = F_n(x_n),\quad n\in \N_0,  \quad x_0 \in \R^d,
\end{equation}
where $F_n \in \cC^1(\R^d,\R^d)$. With  each  trajectory $(x_n)_{n\in\N_0}$ 
we associate the variational equation  
\begin{equation}\label{vari}
u_{n+1} = A_n u_n,\quad n\in\N_0,\quad \text{with}\quad A_n =
DF_n(x_n)\text{ invertible}
\end{equation}
and ask for its angular spectrum. In \cite{behu24b} we introduced the notion of an
angular spectrum for general linear nonautonomous dynamical systems. For every fixed $s < d$
the angular spectrum of dimension $s$ measures the long-time rotational behavior of all 
$s$-dimensional subspaces as they evolve under the linear system \eqref{vari}; see  Definition
\ref{def2:angspec} below.

Our aim is to visualize how the long-time rotational behavior
varies across phase space. Think of leaves floating in a current,
as illustrated in Figure \ref{idea}. We quantify the rotation that these
leaves experience on average along their trajectories by considering
the angular spectra associated with a multitude of initial values.

We make the dependence on the initial value explicit by writing
\begin{equation} \label{eq2:init}
  x_n =x_n(x_0), \quad A_n= A_n(x_0)=DF_n(x_n(x_0)), \quad n \in \N_0.
\end{equation}
Correspondingly, we let $\Phi(n,m;x_0)$ denote the solution operator of \eqref{vari}, defined by
$\Phi(n,m;x_0)=A_{n-1}(x_0) \cdot \ldots \cdot A_m(x_0)$ for $n >m$, $\Phi(n,n;x_0)=I_d$, and
$\Phi(n,m;x_0)=A_n^{-1}(x_0) \cdot \ldots \cdot A_{m-1}^{-1}(x_0)$ for $n <m$.

\begin{figure}[hbt]
\begin{center}
\includegraphics[width=0.90\textwidth]{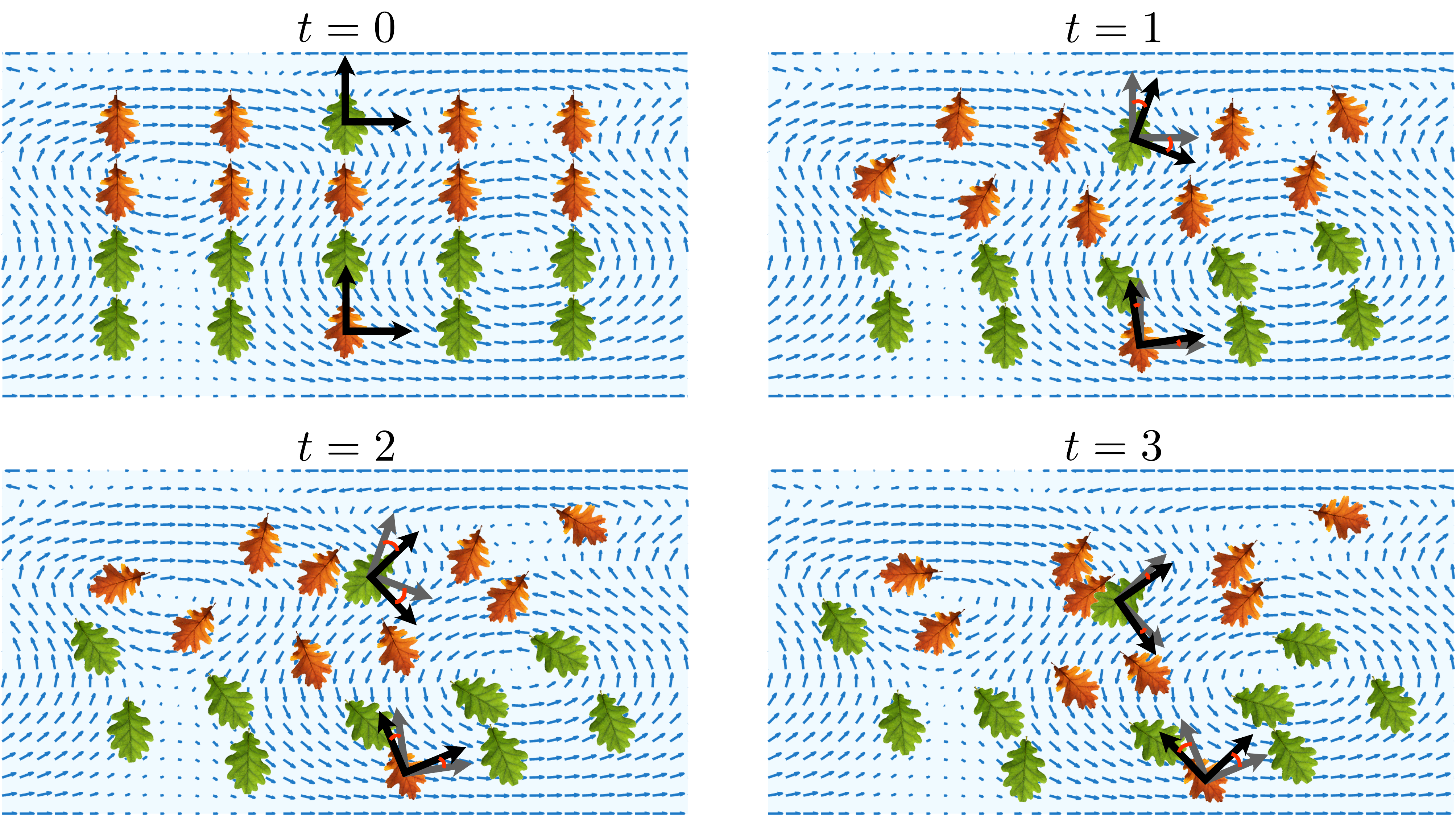}      
\end{center}
\caption{Leaves floating in a vector field. For two selected leaves,
  orthogonal coordinate frames indicate their orientations at the
  current time (black) and at the previous time (gray).  
  Red arcs indicate the rotation angles. \label{idea}}
\end{figure}

Our objects of study live in the Grassmannian 
\begin{equation*} \label{eq1.4}
  \mathcal{G}(s,d) = \{ V \subseteq \R^d: V \text{ is a subspace of
    dimension } s \},
\end{equation*}
which is a smooth manifold of dimension $s(d-s)$ and a metric space w.r.t.\ the
largest principal angle $\ang(V,W)\in [0,\frac{\pi}{2}]$ between two elements $V,W \in \cG(s,d)$. We recall
this notion in  Section \ref{sec1} and summarize some of its basic properties.

The solution operator of \eqref{vari}, \eqref{eq2:init} is then used to form the average of principal angles between iterates of an element $V\in \cG(s,d)$
\begin{equation} \label{def:alpha}
\alpha_s^{\disc}\begin{array}{rcl}
\N \times \cG(s,d) \times \R^d & \to & [0,\frac \pi 2]\\[1mm]
(n,V;x_0) & \mapsto &\displaystyle \frac 1n \sum_{j=1}^n\ang(\Phi(j-1,0;x_0)V,\Phi(j,0;x_0)V).
\end{array}
\end{equation}
These are the key quantities which determine the angular spectrum.
\begin{definition} \label{def2:angspec}
  For $ V \in \cG(s,d)$ and $x_0 \in \R^d$ the interval
  \begin{equation} \label{eq2:angrange}
    I_s^{\disc}(V;x_0) = \left[\varliminf_{n\to \infty} \alpha_s^{\disc}(n,V;x_0), \varlimsup_{n\to \infty}\alpha_s^{\disc}(n,V;x_0)\right]
  \end{equation}
  is called the $s$-dimensional \textbf{angular range} of $V$ for the system \eqref{vari}. The $s$-dimensional 
\textbf{angular spectrum} is defined by
\begin{equation*} \label{eq2:defdisc}
\Sigma_s^{\disc}(x_0) = \mathrm{cl}\Big[ \bigcup_{ V\in\cG(s,d)} I_s^{\disc}(V;x_0) \Big] .
\end{equation*}
The map that assigns to any $x_0$ in some domain $\cB \subseteq \R^d$ the closed set $\Sigma_s^{\disc}(x_0)$
is called the \textbf{angular map} of \eqref{vari} in $\cB$.
\end{definition}

Note that $\Sigma_s^{\disc}$ is called \textit{outer} angular spectrum in
\cite{behu24b}.  We omit the term \textit{outer} in this paper, since other 
variants of the angular spectrum  (see \cite[Section 5.3]{behu24b})
are more difficult to access theoretically and extremely costly to compute
numerically. The superscript ``$\disc$''
is used to distinguish the angular spectrum in discrete time from its
continuous time counterpart. 

The continuous time angular spectrum $\Sigma_s^{\cont}$ is discussed in Section \ref{sec4}. According to \cite{behu24a}
it is based  on the following quantities, defined for $T>0$  and  $V \in \cG(s,d)$:
\begin{equation} \label{eq1:defcont}
     \alpha_s^{\cont}(T,V;x_0)= \frac{1}{T} \int_0^T \|(I_d- P_{\Phi(t,0;x_0)V})A(t;x_0) P_{\Phi(t,0;x_0)V}\| dt.
\end{equation}
Here $\Phi(t,s;x_0)$ is the solution operator of a linear
continuous time system
\begin{equation*}\dot{u}=A(t;x_0)u, \quad u(s)=u_s,
\end{equation*}
which depends on the initial value $x_0$ of a nonlinear system; see \eqref{eq3:contnonlin}, \eqref{eq3:contvari} below.
Furthermore, the symbol $P_V$ in \eqref{eq1:defcont} denotes the orthogonal projector onto
the subspace $V \in \cG(s,d)$, and $\|\cdot\|$ is the spectral norm. The integrand measures the instantaneous rate of rotation of
the evolving subspace. The continuous angular spectrum
$\Sigma_s^{\cont}$ is then defined in an analogous fashion by considering the limit $T \to \infty$; see
Definition \ref{def3:contang}.

Our planar flow example illustrates how angular maps reveal differences
in the mean rotation of infinitesimal perturbations that are not
apparent from the mean rotation along the flow direction alone.
In particular, the latter may lie strictly between the minimal and
maximal angular values; see Figure \ref{Fstrom4}.
\medskip 

We develop theoretical and numerical foundations of angular maps
through the following contributions:
\begin{itemize}
\item[-] We extend the algorithm from \cite{behu24b} to compute
  angular maps. By color-coding
  the minimal and maximal spectral values within the region of
  interest, we obtain a global map of the rotational dynamics.
\item[-] We develop efficient algorithms based on forward and
  backward iteration of subspaces. These techniques apply to several
  examples and avoid the costly computation of the dichotomy spectrum
  and the corresponding spectral bundles.
For generic initial subspaces, we establish conditions under which the angular averages computed by forward iteration asymptotically agree with those associated with the fastest spectral subspaces.
\item[-] We establish a uniform error estimate between the continuous
  angular averages $\alpha_s^{\cont}(T,V;x_0)$ and their suitably
  scaled discrete counterparts $\alpha_s^{\step}(h,T,V;x_0)$,
  obtained from exact flow values at time steps of length $h$.
  This estimate yields convergence of the corresponding angular
  spectra to $\Sigma_s^{\cont}$ in the Hausdorff metric as
  $h \to 0$.
\end{itemize}
\medskip
  
The remainder of the paper is organized as follows.

In Section \ref{sec1} we introduce numerical algorithms based on
forward and backward subspace iteration and illustrate them using
various versions of the H\'enon map.

In Section \ref{AMdiscrete} we relate these algorithms to the reduction theory
developed in \cite{BeHu22}. This theory uses the exponential dichotomy
spectrum of the variational equation and its associated spectral
fibers to reduce the computation of angular spectra to a distinguished
family of subspaces, the so-called trace spaces.
It allows us to identify which parts of the angular spectrum are
captured by the algorithms of Section \ref{sec1}.
We then show how forward and backward iteration can be combined
to recover further spectral fibers and their associated angular
values. The reduction to trace spaces also provides the basis
for a general algorithm for computing angular maps, including
approaches for determining the smallest and largest angular values.
We illustrate these ideas further using the area-preserving
H\'enon map.

In Section \ref{sec4} we extend our approach to continuous-time
systems and illustrate it for a planar flow and the
three-dimensional Lorenz system. We establish the approximation
and convergence results stated above (Theorems
\ref{th3:unifapproxT} and \ref{th4:approxspechaus}).
Finally, we present and analyze a simplified algorithm for
autonomous systems that uses successive trajectory points to
approximate the angular range associated with the flow direction.
 

\section{Principal angles and algorithms for their numerical approximation}\label{sec1} 
The maximal principal angle  of two subspaces $V,W \in \cG(s,d)$ is denoted by $\ang(V,W)$.
It can be characterized as follows (see \cite[Prop.\ 2.3]{BeFrHu20})
\begin{equation*}\label{A1}
    \ang(V,W) =
    \max_{\substack{v\in V \\ v\neq 0} }\min_{\substack{w\in W \\ w \neq 0}} \ang(v,w)
  = \arccos\big(\min_{\substack{v\in V\\\|v\|=1}} \max_{\substack{
      w\in W\\\|w\|=1}} v^{\top} w\big).
\end{equation*}
Note that  $\ang(v,w) =\ang(\Span(v),\Span(w))= \arccos\left(\frac{|v^{\top}w|}{\|v\| \|w\|}\right) \in [0,\tfrac{\pi}{2}]$
measures the principal angle between the subspaces $\Span(v)$ and $\Span(w)$. It ignores the 
orientation of their spanning vectors $v$ and $w$ measured by $\angle(v,w)=\arccos\left(\frac{v^{\top}w}{\|v\| \|w\|}\right) \in [0,\pi]$. For general subspaces $V,W\in  \cG(s,d)$, let the columns of the
$d \times s$-matrices $V_B$ and $W_B$  form an orthonormal basis of $V$ resp.\ $W$; then
the maximal principal angle satisfies
$\ang(V,W)=\mathrm{arccos}(\sigma_s)$, where $\sigma_1\ge \sigma_2\ge
\cdots\ge \sigma_s\ge 0$ denote the singular values of
$V_B^{\top}W_B$; see \cite[Ch.\ 6.4.3]{GvL2013}. 
The maximal principal angle defines a metric on $\cG(s,d)$ and so does the expression
\begin{equation} \label{eq2:charsin}
  \| P_V - P_W\|= \sin(\ang(V,W)), \quad V,W \in \cG(s,d),
\end{equation}
where $\| \cdot \|$ denotes the spectral norm and $P_V,P_W$ are the orthogonal projectors onto
the subspaces $V$  and $W$, respectively  \cite[Ch.\ 6.4.3]{GvL2013}, \cite[Prop.\ 2.3]{behu24a}. A further equation relates the angle $\ang(V,W)$ to the distance of the unit ball
in one subspace from the other subspace; see \cite[Ch.II.4]{StSu1990}:
\begin{equation} \label{eq2:chardist}
  \sup\{ \dist(v,W): v \in V, \|v\|=1 \}=\sin(\ang(V,W)).
  \end{equation}
Similarly, one finds for the minimal angle $ \theta_{\min}(V,W)= \arccos(\sigma_1)$
between $V$ and $W$
\begin{equation} \label{eq2:charmin}
  \inf\{ \dist(v,W): v \in V, \|v\|=1 \}=\sin(\theta_{\min}(V,W)).
  \end{equation}
Sometimes, it is useful to work with orthogonal complements of subspaces rather than the subspaces
themselves (see Algorithm \ref{alg2adj}), using the following Lemma.
\begin{lemma} \label{lem1:adjoint}
  For $V, W \in \cG(s,d)$ and nonsingular $A \in \R^{d,d}$ the following relations hold:
  \begin{equation} \label{eq:perpequal}
    \ang(V,W) = \ang(V^{\perp}, W^{\perp}),
  \end{equation}
  \begin{equation*} \label{eq:itadj}
    \ang(V,AV) = \ang(V^{\perp},A^{-\top}V^{\perp}).
    \end{equation*}
\end{lemma}
\begin{proof}
  The equality \eqref{eq:perpequal} follows from \eqref{eq2:charsin} and the fact that
  $I-P_V = P_{V^{\perp}}$ is the projector onto the orthogonal complement:
  \begin{equation*}
    \sin(\ang(V,W))=\|P_V - P_W\|= \|(I_d - P_V)- (I_d-P_W)\| =\sin(\ang(V^{\perp},W^{\perp})).
  \end{equation*}
  Further, the algebraic identity $(AV)^{\perp}=A^{-\top}V^{\perp}$ and \eqref{eq:perpequal} yield
  \begin{equation*}
    \ang(V,AV)= \ang(V^{\perp},(AV)^{\perp})= \ang(V^{\perp},A^{-\top}V^{\perp}).
  \end{equation*}
  {  }
\end{proof}
 
\subsection{One-dimensional spectra in discrete time -- the fast subspace}
\label{sec1.1}
For the discrete time system \eqref{nonlin}, we propose a simple
algorithm to illustrate the rotational dynamics of several trajectories.
The area of interest is given for dimension $d=2$ resp.\ $d=3$ by
\[
\cB^2 = [a_-,a_+] \times [b_-,b_+] \quad \text{and} \quad
\cB^3 = [a_-,a_+] \times [b_-,b_+] \times [c_-,c_+].
\]
Let $L\in\N$ denote the resolution of the grid and define
\[
h_a = \frac {a_+-a_-}L,\quad 
h_b = \frac {b_+-b_-}L \quad \text{and in the case }d=3\quad 
h_c = \frac {c_+-c_-}L.
\]
For $d=2$ we define the $(i,j)$-th box for $i,j\in\{1,\dots,L\}$ as
\[
\cB_{i,j} = [a_-+(i-1)h_a,\ a_-+ih_a] \times [b_-+(j-1)h_b,\ b_-+j h_b].
\]
In Algorithm \ref{alg1b} we let  $N\in\N$ denotes the trajectory length
and assume that the Jacobians $DF_{n-1}$ are available.  
The algorithm aims to compute the average rotation angle
w.r.t.\ the fastest direction. 
\begin{algorithm}
\caption{Average rotation angle w.r.t.\ the fastest direction ($s=1$) 
 \label{alg1b}}
\textcolor{header1}{
\begin{algorithmic}[1]
\For {$i = 1,\dots,L$}
\For {$j = 1,\dots,L$}
\For {$k=1,\dots,L$ (if $d =3$)}
\State Choose the midpoint $x_0 \in \cB_{i,j,(k)}$ 
\State Choose $V_0 \in \R^d$, $\|V_0\|=1$ at random 
\For {$n = 1,\dots,N$}
\State $x_{n} = F_{n-1}(x_{n-1})$
\State $Y = DF_{n-1}(x_{n-1}) V_{n-1}$
\State $V_n = \frac Y {\|Y\|}$
\EndFor
\State $\displaystyle\Gamma_{i,j,(k)} = \frac 1N \sum_{\ell =1}^N \ang(V_{\ell -1},V_\ell)$
\If {$x_n \in \cB\ \forall n\in\{1,\dots,N\}$}
\State Color the box $\cB_{i,j,(k)}$ with the color
that corresponds to the angle $\Gamma_{i,j,(k)}$  
\EndIf
\EndFor
\EndFor
\EndFor
\end{algorithmic}}
\end{algorithm}

In Algorithm \ref{alg1b} we include the case $d=3$ and define the $(i,j,k)$-th box for $i,j,k\in\{1,\dots,L\}$ as
\[
\cB_{i,j,k} = [a_-+(i-1)h_a,\ a_-+ih_a] \times [b_-+(j-1)h_b,\ b_-+j h_b]
 \times [c_-+(k-1)h_c,\ c_-+k h_c].
\] 

We apply Algorithm \ref{alg1b} to the following two- and three-dimensional
  H\'enon systems:
\begin{equation}\label{henonmaps}
f_2:\begin{array}{rcl}
\R^2 & \to & \R^2\\
\begin{pmatrix}x_1\\x_2 \end{pmatrix}
& \mapsto &
\begin{pmatrix}
1+x_2-1.4 x_1^2\\ 0.3 x_1
\end{pmatrix}
\end{array}
\quad \text{and} \quad
f_3 :\begin{array}{rcl} 
\R^3&\to & \R^3\\
x & \mapsto & 
\begin{pmatrix} 1+ x_3 - 1.4 x_1^2\\x_1+x_3\\0.2 x_1 + 0.1 x_2  
\end{pmatrix}.
\end{array}
\end{equation}
For the two-dimensional H\'enon map, we consider the box $\cB =
[-1.5, 1.5]^2$ with resolution $L = 10^3$.  We use 
 $N = 30$   iterates for the left diagram in Figure \ref{Fh1} and $N = 10^4$
for the right diagram. 
In addition, we plot the H\'enon attractor in black to indicate its position within the domain of colored points
that do not leave the prescribed box. 

\begin{figure}[hbt]
\begin{center}
\includegraphics[width=0.99\textwidth]{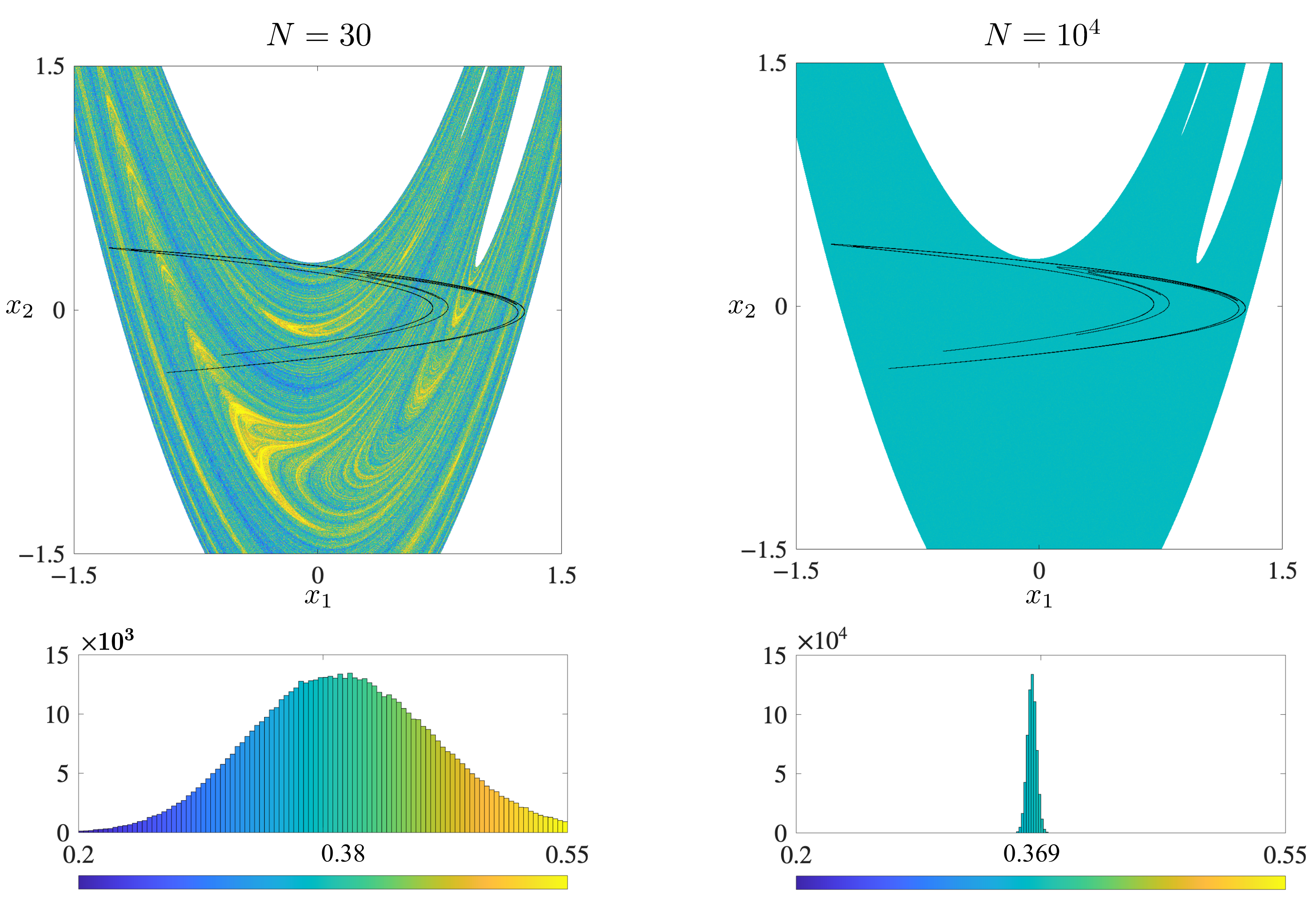}      
\end{center}
\caption{Algorithm \ref{alg1b} applied to $f_2$  (see \eqref{henonmaps})
  with $N = 30$ (left) and $N=10^4$ (right). \label{Fh1}} 
\end{figure}
In the second row of the figure we provide a histogram of approximate
angular spectral  data for those initial values among the $L^2=10^6$
grid points whose trajectories remain in $\cB$.  
The color coding of points in phase space is synchronized with their corresponding
position in the histogram. 
We observe a marked difference between the results for short
($N=30$) and long trajectories ($N=10^4$). For long trajectories, angular spectral values show
a very narrow distribution and almost no dependence on the initial points in the domain of
attraction. However,  for short trajectories, we observe a broad distribution function with high
values (yellow) and low values (blue) originating from different parts of the domain of attraction.  

For the three-dimensional H\'enon map, we choose the box $\cB =
[-2, 2]\times [-3,3] \times [-3,3]$ with resolution $L = 300$.  We use
 $N = 30$ steps for the left diagrams in Figure \ref{Fh2} and $N = 10^4$
for the right diagrams. These diagrams look very similar to the $2D$ case: a broad distribution around a specific angular value
related to a  structured domain of attraction and a very narrow distribution  centered at
$\alpha_1^{\disc}=0.836$ almost uniformly for all initial values in
the domain of attraction. In Section \ref{AMdiscrete} below we show
where this angular spectral value is located in the
angular spectrum.

\begin{figure}[H]
\begin{center}
\includegraphics[width=0.99\textwidth]{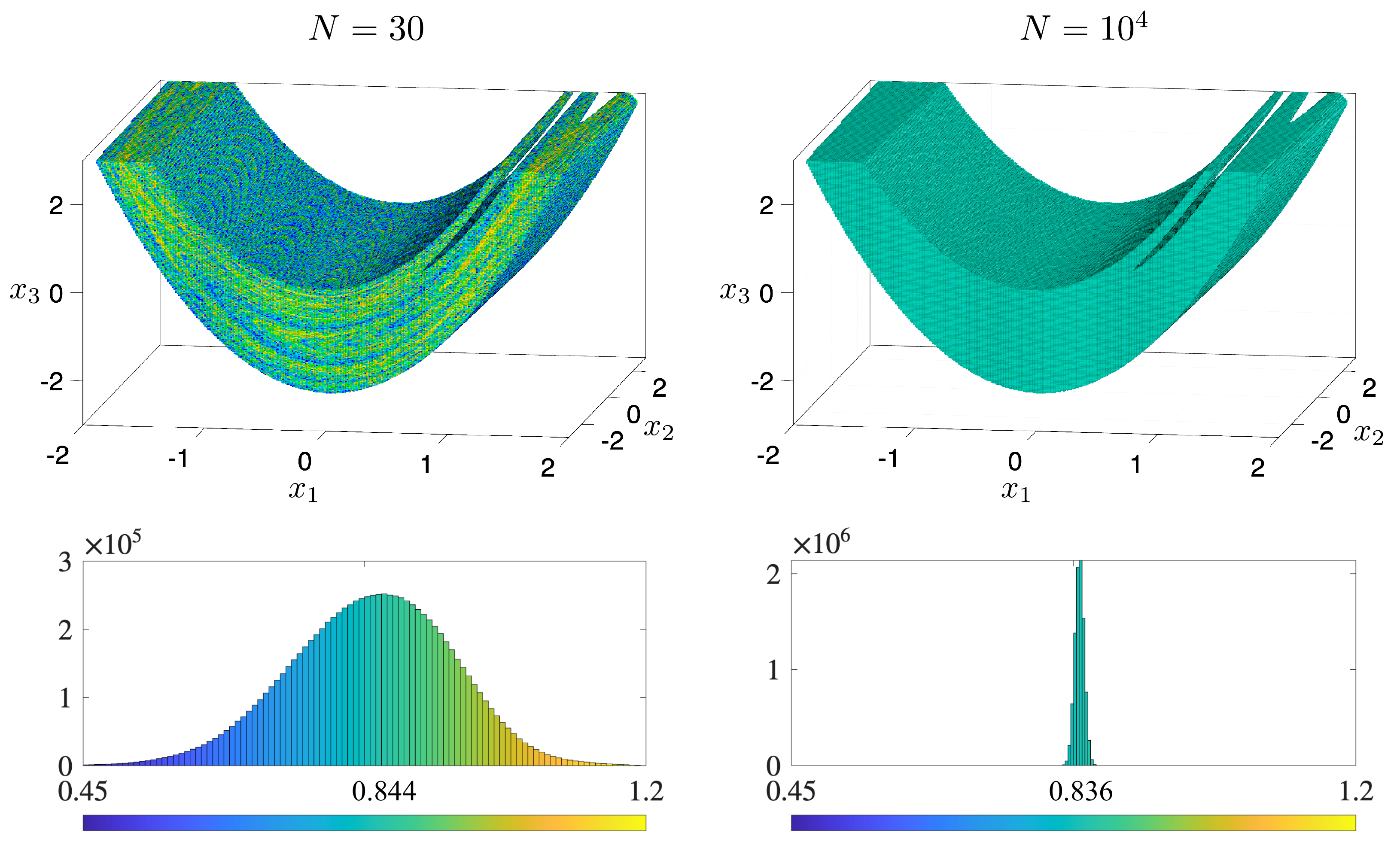}      
\end{center}
\caption{Algorithm \ref{alg1b} applied to $f_3$  (see \eqref{henonmaps})
  with $N = 30$ (left) and $N=10^4$ (right). \label{Fh2}} 
\end{figure}
  
\subsection{ A $QR$-like algorithm for multidimensional angular maps}
\label{sec1.2}
We generalize Algorithm \ref{alg1b}
 to the general case $s\ge 1$. For this step we use $QR$ factorization analogous to the transition from the power
 method to subspace iteration for eigenvalue problems; see \cite[Ch.\
 7.3, 7.6]{GvL2013}.
\begin{algorithm}
\caption{$QR$-like algorithm replacing lines 5--10 of
  Algorithm~\ref{alg1b} ($s\ge 1$) \label{alg2}}
\textcolor{header1}{
\begin{algorithmic}[1] 
\makeatletter
\setcounter{ALG@line}{4}
\makeatother
\State Choose $V_0 \in \R^{d,s}$, $V_0^{\top}V_0=I_s$ at random
\For {$n = 1,\dots,N$}
\State $x_{n} = F_{n-1}(x_{n-1})$
\State $Y = DF_{n-1}(x_{n-1})V_{n-1}$
\State $(V_n,R_n) = QR(Y)$
\EndFor
\State $\displaystyle\Gamma_{i,j,(k)} = \frac 1N \sum_{\ell =1}^N \ang(V_{\ell -1},V_\ell)$
\end{algorithmic}}
\end{algorithm}
In line 9 we use the thin $QR$ factorization for which $V_n\in \R^{d,s}$,
$V_n^{\top}V_n =I_s$ and $R_n \in \R^{s,s}$ is upper triangular; see  \cite[Theorem 5.2.3]{GvL2013}.
For large $N$, the value $\Gamma_{i,j,(k)}$ approximates some value in the angular range
$I_s^{\disc}(V_0;x_0)$. In Section \ref{AMdiscrete} we will discuss how the specific value
selected by the algorithm  relates to the invariant foliation induced by
the dichotomy spectrum of the linearized equation \eqref{vari}.
 
We apply Algorithm \ref{alg2} to the three-dimensional
H\'enon system $f_3$  (see \eqref{henonmaps}) in case $s=2$, using the
setup from Figure \ref{Fh2}. We note that computing the
QR factorizations is considerably more time-consuming than the
normalization performed in line 9 of Algorithm~\ref{alg1b}.
We present the histograms for  $N = 30$ and $N = 10^4$ in
Figure \ref{Fh3}, which provide the most relevant information.
\begin{figure}[H]
\begin{center}
\includegraphics[width=0.90\textwidth]{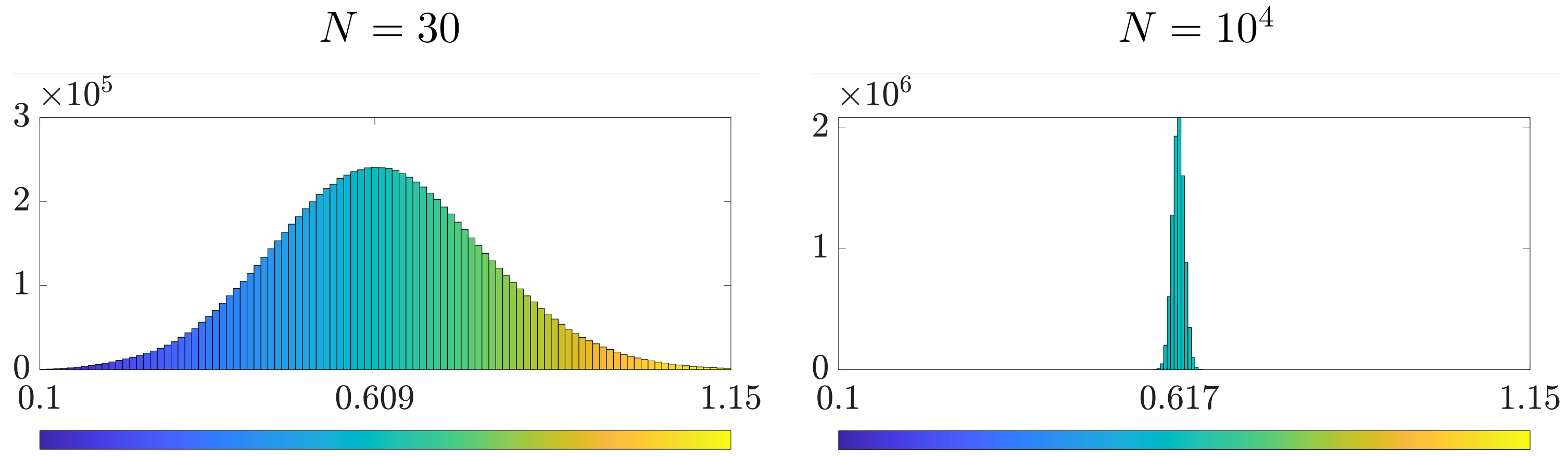}      
\end{center}
\caption{Algorithm \ref{alg2} applied to $f_3$ (see \eqref{henonmaps})
 for $s=2$ with $N = 30$ (left) and $N=10^4$ (right). \label{Fh3}} 
\end{figure}
The orthogonal complements of elements $V\in \cG(s,d)$ satisfy $V^{\perp} \in \cG(d-s,d)$.
Therefore, if $s> \tfrac{d}{2}$, it is more efficient to iterate orthogonal complements
of smaller dimension and employ  Lemma \ref{lem1:adjoint} for computing angles.
This leads to the following algorithm:
\begin{algorithm}
\caption{$QR$-like algorithm for orthogonal complements, \\ replacing line 5--10 of
  Algorithm~\ref{alg1b} ($s\ge 1$) \label{alg2adj}}
\textcolor{header1}{
\begin{algorithmic}[1] 
\makeatletter
\setcounter{ALG@line}{4}
\makeatother
\State Choose $U_0 \in \R^{d,d-s}$, $U_0^{\top}U_0=I_{d-s}$ at random
\For {$n = 1,\dots,N$}
\State $x_{n} = F_{n-1}(x_{n-1})$
\State Solve $ DF_{n-1}(x_{n-1})^{\top}Z=U_{n-1}$
\State $(U_n,S_n) = QR(Z)$
\EndFor
\State $\displaystyle\Gamma_{i,j,(k)} = \frac 1N \sum_{\ell =1}^N \ang(U_{\ell -1},U_\ell)$
\end{algorithmic}}
\end{algorithm}

If we choose $U_0\in \R^{d,d-s}$ in Step $5$ such that $U_0^{\top}V_0=0$ and  $U_0^{\top}U_0=I_{d-s}$  holds for
the initial space $V_0$ from Algorithm \ref{alg2}, then both algorithms yield the same value, i.e.,
\begin{equation*} \label{eq:UV}
  \sum_{\ell=1}^N \ang(V_{\ell-1},V_{\ell})=\sum_{\ell =1}^N \ang(U_{\ell -1},U_\ell).
\end{equation*}
  Indeed, by induction we obtain
$U_n^{\top}V_n=0$ from the invertibility of $R_n,S_n$ and
\begin{equation*}
  S_n^{\top}U_n^{\top}V_nR_n=U_{n-1}^{\top} DF_{n-1}(x_{n-1})^{-1}DF_{n-1}(x_{n-1})V_{n-1}=
  U_{n-1}^{\top}V_{n-1} =0.
\end{equation*}
The equality 
 $\ang(V_{\ell-1},V_{\ell})=\ang(U_{\ell-1},U_{\ell})$, $\ell=1,\ldots,N$, then follows from \eqref{eq:perpequal}.

For the second angular value of the $3D$-H\'{e}non map we have $s=2>\tfrac{3}{2}$.
Hence, it is more efficient to iterate normal vectors and normalize
them at each step rather than
iterate bases of two-dimensional subspaces and perform a $QR$ decomposition at each step.
This reduces the computation time by a factor of $7$. If, in addition,
an explicit formula for the inverse Jacobian is used, the computation
time is reduced by another factor of $3$.

\subsection{Application to inverted maps}
\label{sec1.3}
Until now, we have computed angular maps only through forward iteration.
Forward subspace iteration captures angular values associated with the
fastest directions, while backward iteration captures those associated
with the slowest directions. 
For this purpose we consider an
invertible map $f: \R^d \to \R^d$ and the autonomous system
\begin{equation} \label{inverseit}
  y_{n+1}=f^{-1}(y_n), \quad  n \in \N_0, \quad y_0 \in \R^d.
  \end{equation}
The associated  variational equation \eqref{vari} takes the form 
\begin{equation} \label{invvar}
u_{n+1} = (Df(y_{n+1}))^{-1} u_n,\quad n\in \N_0,
\end{equation}
since $ Df^{-1}(y_{n})=(Df(f^{-1}(y_{n})))^{-1}=(Df(y_{n+1}))^{-1} $. 

Numerically, we avoid iterating the inverse map in \eqref{inverseit}
and instead use the orbit segment $(x_n)_{n\in\{0,\dots,N\}}$ which
has  already been computed forward in time.   
For the initial value $y_0 = x_N$, we get the solution of
\eqref{inverseit} by $y_n = x_{N-n}$ for $n\in\{0,\dots,N\}$. 
Further setting $v_k=u_{N-k}$, $k \in \{0,\ldots,N\}$, the finite time  variational equation \eqref{invvar} takes
the form 
\[
v_{k-1} = (Df(x_{k-1}))^{-1} v_k,\quad k =N,\ldots,1.
\]
Figure \ref{Fh4} shows the histograms of the forward and backward
angular maps for the two-dimensional H\'enon map $f_2$ 
in a single diagram. We choose  $N = 1000$ with the
remaining parameters  as described above.

\begin{figure}[hbt]
\begin{center}
\includegraphics[width=0.60\textwidth]{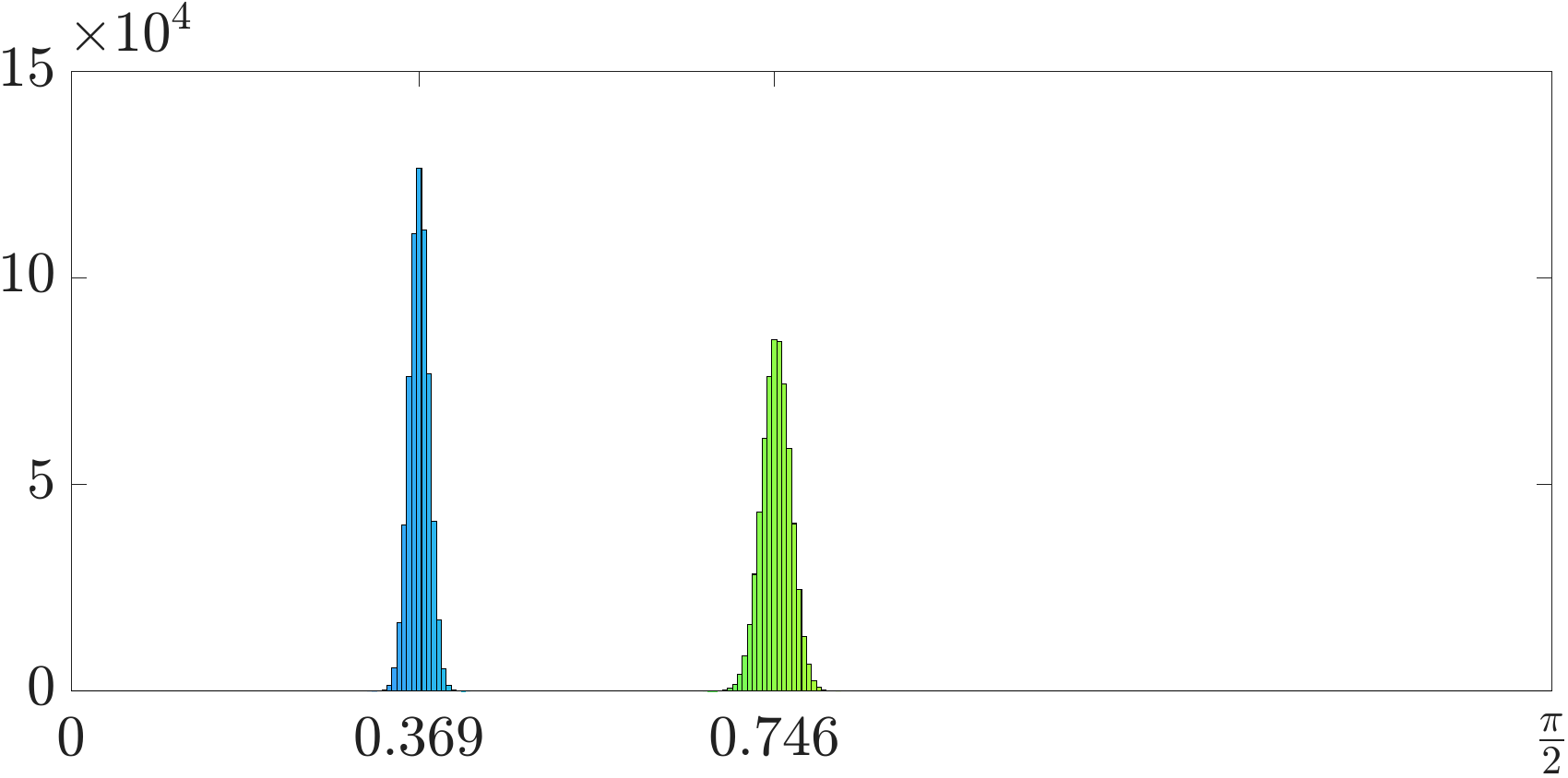}      
\end{center}
\caption{Histogram of angular maps for $f_2$ with $N=1000$. The left
  peak corresponds to the forward computation, and the right peak to
  the backward computation. \label{Fh4}}  
\end{figure}

We repeat this computation in Figure \ref{Fh5} for the
three-dimensional H\'enon map.

\begin{figure}[hbt]
\begin{center}
\begin{tabular}{cc}
$s = 1$ & $s=2$\\[-1mm]
\includegraphics[width=0.47\textwidth]{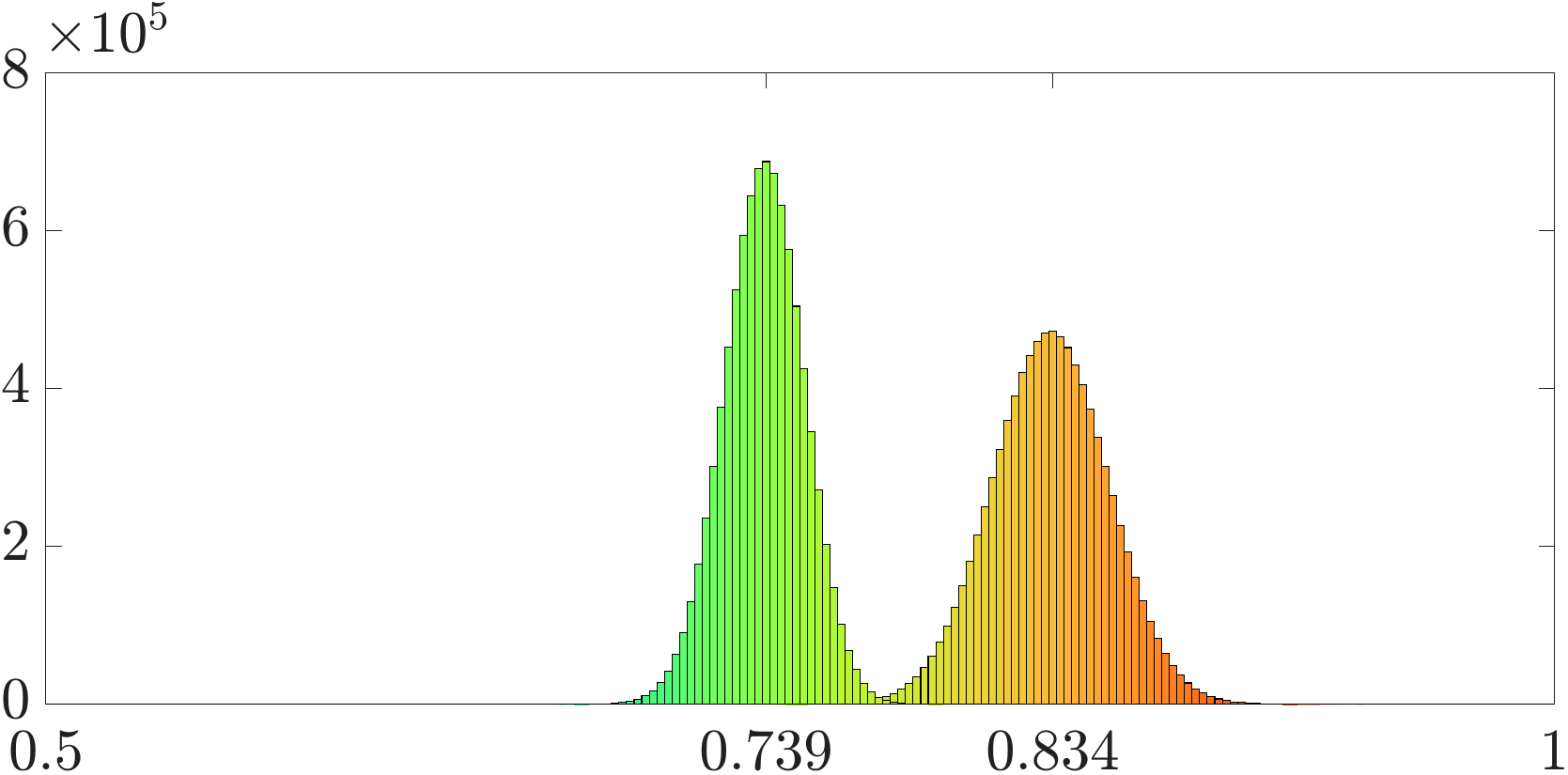}&
\includegraphics[width=0.47\textwidth]{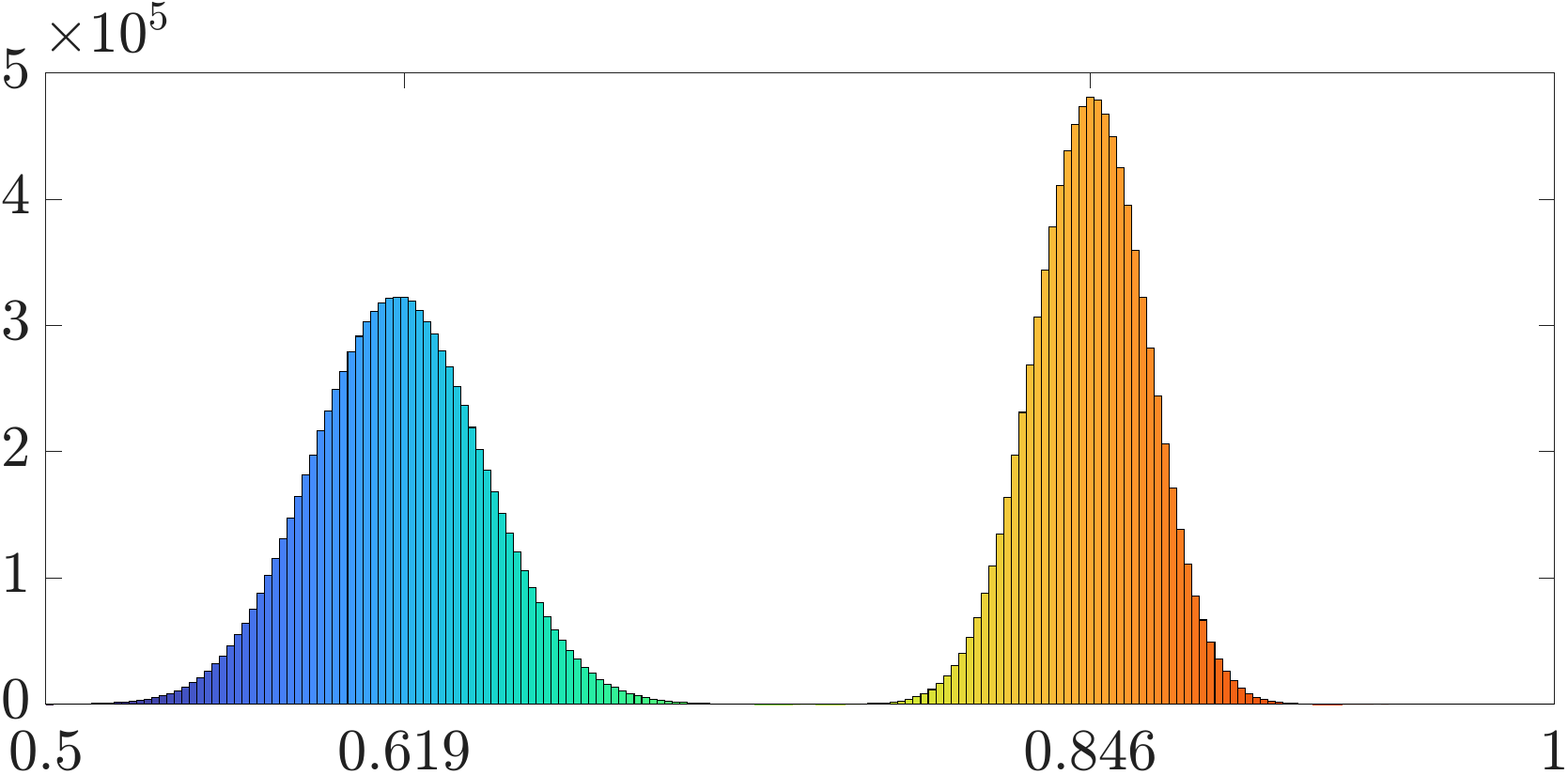}
\end{tabular}      
\end{center}
\caption{Histogram of angular maps for $f_3$ with $N=1000$. 
For $s=1$ the left peak corresponds to the backward computation, and
the right peak to the forward computation. For $s=2$ the left peak
corresponds to the forward computation, and the right peak to the
backward computation. \label{Fh5}}  
\end{figure}

In the next section, we explain why forward and backward iteration can
capture the entire angular spectrum for the two-dimensional H\'enon
map. Our numerical results suggest that, for typical trajectories
approaching the attractor, this spectrum consists of the two values
approximated by the peaks in Figure \ref{Fh4}, with no apparent dependence
on the initial point. Exceptional trajectories, such as fixed points
and periodic orbits on the attractor, may have different angular
spectra. For the three-dimensional H\'enon map, however, forward and
backward iteration alone miss one of the three spectral values. 

\section{Angular maps in discrete time and a general algorithm}\label{AMdiscrete}
In this section, we show that the algorithms discussed so far capture only a specific
part of the angular spectrum and thus of the angular map. This follows from a
key result of 
the reduction theory developed in \cite{BeHu22}, which is based on the exponential dichotomy
spectrum of the variational equation \eqref{vari}.
As a consequence, we will use the methods from \cite{BeHu22} to develop an algorithm
capable of computing the entire angular spectrum. However, the computational cost
of this algorithm is orders of magnitude higher than the simple algorithms considered so far.

\subsection{Angular range and dichotomy spectrum}
\label{sec3.1}
The notion of an ``angular range'' is justified by the following Lemma taken from
\cite[Lemma 2.6]{behu24b}.
\begin{lemma} \label{lem2:accu}
  The angular range  $ I_s^{\disc}(V;x_0)$ is the set of  accumulation points of  the sequence
  $(\alpha_s^{\disc}(n,V;x_0))_{n \in \N}$.
\end{lemma}
 
We use the reduction theory from \cite[Section 3]{BeHu22} to analyze which part
of the angular spectrum is obtained by applying the $QR$-like Algorithm \ref{alg2}
to a generic subspace $V \in \cG(s,d)$.

Let us assume that the matrices $A_n$ in \eqref{vari} are uniformly bounded and have
uniformly bounded inverses. According to the  spectral theorem \cite[Theorem 5.4]{as01}
for exponential dichotomies (see also \cite{p10}),
there exists an index $\varkappa\le d$ and numbers 
\begin{equation*} \label{eq3:dichspec}
 0 = \sigma_{\varkappa+1}^+ < \sigma_{\varkappa}^{-} \le \sigma_{\varkappa}^{+} < \sigma_{\varkappa-1}^{-}\le \sigma_{\varkappa-1}^{+}< \ldots < \sigma_1^{-} \le \sigma_1^{+}
\end{equation*}
and projectors $P_n^i$, $n \in \N$, $i=1,\ldots, \varkappa+1$, with the following properties:
\begin{itemize}
  \item[(i)]
For all $0 \le m \le n$, $i=1,\ldots,\varkappa$ there holds
\begin{equation}  \label{eq3:projprop}
  \begin{aligned}
    \Phi(n,m)P_m^i&=P_n^i\Phi(n,m), \qquad \; \; \; \text{(invariance)} \\
    \{0\}= \range(P_n^{\varkappa+1})& \subseteq \range(P_n^{\varkappa}) \subseteq  \cdots \subseteq \range(P_n^1)= \R^d,\;\;\;  \text{(flag of ranges)} \\
    \R^d = \kernel(P_n^{\varkappa+1})& \supseteq \kernel(P_n^{\varkappa}) \supseteq \cdots \supseteq
      \kernel(P_n^1) = \{0\}. \; \; \; \text{(flag of null spaces)}
     \end{aligned}
\end{equation}
\item[(ii)]
For each $i=1,\ldots,\varkappa$ and for every $0< \eps < \sigma_i^--
\sigma_{i+1}^+ $ there exists a constant $K_i(\eps)$ such that for all
$0 \le m \le n$ there holds
\begin{equation}\label{eq3:dichprop}
  \begin{aligned}
    \|\Phi(n,m)P_m^{i+1} \| & \le K_i(\eps) (\sigma_{i+1}^{+}+\eps)^{n-m},
                          \quad  \text{(forward dichotomy estimate)}\\ 
    \|\Phi(m,n)(I_d-P_n^{i+1}) \| & \le K_i(\eps) (\sigma_{i}^{-}-\eps)^{m-n}. \quad 
    \text{(backward dichotomy estimate)}
  \end{aligned}
\end{equation}
\end{itemize}
The set
\begin{equation*}
  \Sigma_{\mathrm{ED}}= \bigcup_{i=1}^{\varkappa} [\sigma_i^{-},\sigma_i^{+}]
\end{equation*}
is called the dichotomy spectrum of \eqref{vari}.
\begin{remark} \label{bem3:cautionepsilon} In \cite[Section 3]{BeHu22} we derived
  the dichotomy estimates from the exponential dichotomy of a scaled operator. However, our conclusion that  the estimates \eqref{eq3:dichprop} hold with $\eps=0$ is  only
  true under additional assumptions. In general, the estimates are satisfied 
  for sufficiently small $\eps>0$  with rates and
  constants depending on $\eps$.
  With this modification,  \cite[Theorem 3.4]{BeHu22} applies  to 
  $\eps=\frac{1}{4}\min_{i=1,\ldots,\varkappa}(\sigma_i^-- \sigma_{i+1}^+)$ and then
  yields the rate $r<1$ used in \eqref{eq3:Vundtrace} below.
  \end{remark}
Due to the nested subspaces \eqref{eq3:projprop},  the whole space can be decomposed into invariant fibers
\begin{equation} \label{eq3:fibers} \textstyle
  \R^d = \medoplus_{i=1}^{\varkappa} \cW_n^i, \quad \cW_n^i= \range(P_n^i) \cap \kernel(P_n^{i+1}),\quad
  i=1,\ldots,\varkappa.
\end{equation}
Note that we have $\cW_n^i=\range(\cP_n^i)$ for the
fiber projectors $\cP_n^i=(I_d-P_n^{i+1})P_n^i=P_n^i - P_n^{i+1}= P_n^i(I-P_n^{i+1})$.
Because of  the exponential estimates \eqref{eq3:dichprop}, we call $\cW_n^{\varkappa}$ the slowest fiber and $\cW_n^1$ the
fastest fiber.
Further, according to \cite[(3.16)]{BeHu22} we  associate to every $V \in \cG(s,d)$
its trace spaces
\begin{equation} \label{eq3:deftrace} \textstyle
  \cT_n(V) = \medoplus_{i=1}^{\varkappa}W_n^i(V), \quad \text{with}\quad W_n^i(V)= (I_d- P_n^{i+1})(\range(P_n^i)\cap V),
\end{equation}
which satisfy $W_n^i(V) \subseteq \cW_n^i$ as well as $\sum_{i=1}^{\varkappa} \dim(W_n^i(V))=s$.

The main result \cite[Theorem 3.4]{BeHu22} states that the angle between the iterates
of $V$ and its trace spaces $\cT_n(\Phi(n,0)V)=\Phi(n,0)\cT_0(V)$ converges to zero at an exponential rate. More precisely, for
every $V \in \cG(s,d)$ there exists a constant $C(V)$ such that for all $n \ge 0$
\begin{equation} \label{eq3:Vundtrace}
  \ang(\Phi(n,0)V,\Phi(n,0)\cT_0(V)) \le C(V)r^n,  \quad
  r= \frac{1}{2}\Big(1+\max_{i=1,\ldots,\varkappa}
  \frac{\sigma_{i+1}^+}{\sigma_i^-}\Big)  <1.
\end{equation}
As noted above, the  rate $r$ is determined by the smallest gap between the spectral intervals.
As a consequence of \eqref{eq3:Vundtrace}, we find that the angular range of $V$ coincides
with the angular range of its trace space at $n=0$ \cite[Theorem 2.17]{behu24b}:
\begin{equation} \label{eq3:angrangeequal}
  I_s^{\disc}(V) = I_s^{\disc}(\cT_0(V)),  \quad V \in \cG(s,d).
\end{equation}
Hence, one can reduce the computation of  the angular spectrum to the set of trace spaces:
\begin{equation}\label{trace} 
  \begin{aligned} 
    \Sigma_s^{\disc}&= \mathrm{cl}\Big[ \bigcup_{V\in \cD_0(s,d)} I_s^{\disc}(V) \Big], \\  
    \cD_0(s,d)&= \Big\{\textstyle \medoplus_{i=1}^{\varkappa} W^i:  W^i \text{ subspace of } \cW_0^i\ (i=1,\ldots,\varkappa),\ \sum_{i=1}^{\varkappa} \dim W^i= s \Big\}.
    \end{aligned}
\end{equation}
In this way, the computational effort can be reduced significantly. In particular, if all fibers
are one-dimensional, the whole angular spectrum can be computed from the angular
range of  at most $ d \choose s$ trace spaces.
This case frequently occurs with our applications below.

Further, if we choose a random initial subspace as in Algorithm \ref{alg2}, then we expect the algorithm
to converge to the angular range of a trace space composed of the fastest
fibers in \eqref{eq3:fibers}. To see this, determine the unique index $j\in \{1,\ldots,\varkappa\}$
with
\begin{equation} \label{eq3:determinej} \textstyle
  \sum_{i=1}^{j-1} d_i < s \le \sum_{i=1}^{j} d_i= d - \sum_{i=j+1}^{\varkappa}
  d_i, \quad d_i=\dim(\cW_0^i),\quad i=1,\ldots,\varkappa.
\end{equation}
From the inequality $s + \sum_{i=j+1}^{\varkappa}d_i \le d$ we conclude
that the two subspaces
$V$ and $\medoplus_{i=j+1}^{\varkappa}\cW_0^i$ almost surely have a trivial intersection.
Then $W_0^i(V)=\{0\}$ holds for $i=j+1,\ldots,\varkappa$ in \eqref{eq3:deftrace}  and
\begin{equation} \label{eq3:fasttrace} \textstyle
  \cT_0(V) =\medoplus_{i=1}^j W_0^i(V) \subseteq \medoplus_{i=1}^j \cW_0^i.
\end{equation}
Therefore,   the algorithm will generate the angular range of a subspace that belongs to a sum of
the fastest fibers. In case of equality on the righthand side of \eqref{eq3:determinej},
we can say even more:
\begin{proposition} \label{cor3:anglefast}
  Let $0<s <d$ and $j \in \{1,\ldots,\varkappa\}$ be given. Then there exist constants
  $C>0$ and $0\le r<1$  such that for all $V\in \cG(s,d)$ with 
  \begin{equation} \label{eq3:sexact}
    s=\textstyle  \sum_{i=1}^j \dim(\cW_0^i), \quad V \cap \medoplus_{i=j+1}^{\varkappa}\cW_0^i = \{0\},
  \end{equation}
  the following estimates hold for the iterates of $V$ and the
  angular sums from \eqref{def:alpha} 
  \begin{equation} \label{eq3:alphanest}
    \begin{aligned}
       \ang(\Phi(n,0)V,\textstyle \medoplus_{i=1}^j \cW_n^i) & \le
             \frac{C}{\theta_{\min}}r^n, \quad \forall n \ge 1, \\
       |\alpha^{\disc}_s(n,V)- \alpha_s^{\disc}(n,\textstyle \medoplus_{i=1}^j \cW_0^i)| &\le
       \frac{C}{n \theta_{\min}}, \quad \forall n \ge 1.
       \end{aligned}
  \end{equation}
  Here $\theta_{\min}=\theta_{\min}(V, \medoplus_{i=j+1}^{\varkappa}\cW_0^i) $ denotes the minimal principal angle between the  subspaces $V$ and $\medoplus_{i=j+1}^{\varkappa}\cW_0^i$.
  Further, the equality $I_s^{\disc}(V)= I_s^{\disc}(\medoplus_{i=1}^j \cW_0^i)$ holds.
\end{proposition}
\begin{remark} \label{rem2:Cspecific} Note that we have $\theta_{\min}>0$ by formula
  \eqref{eq2:charmin} and assumption \eqref{eq3:sexact}. The estimate \eqref{eq3:alphanest}
  refines the result of \cite[Theorem 3.4]{BeHu22} by specifying the dependence of the constant $C(V)$ on $V$ in the asymptotic estimate \eqref{eq3:Vundtrace}.
  The closer $V$ lies to the sum of the $\varkappa-j$ slow  fibers, the longer
  it takes for the asymptotic regime to dominate. 
    \end{remark}
\begin{proof}
  From the equality of dimensions in \eqref{eq3:sexact} we find equality in
  \eqref{eq3:fasttrace} as well, i.e.
  \begin{equation} \label{eq3:fspaces} \textstyle
    \cT_0(V) = \medoplus_{i=1}^j \cW_0^i= \kernel(P_0^{j+1}), \quad
    \medoplus_{i=j+1}^{\varkappa} \cW_0^i= \range(P_0^{j+1}). 
  \end{equation}
 Moreover, assumption \eqref{eq3:sexact} assures that the map
  \begin{equation*}
    (I_d-P_0^{j+1})_{|V}: V \to \kernel(P_0^{j+1})
  \end{equation*}
  is injective, hence invertible by the equality of dimensions. Setting
  \begin{equation*}
      L_0= P_0^{j+1} \left((I_d-P_0^{j+1})_{|V}\right)^{-1}: \kernel(P_0^{j+1}) \to \range(P_0^{j+1}),
          \end{equation*}
  we can write $V$ as a graph over $\kernel(P_0^{j+1})$, i.e.
  \begin{equation} \label{eq3:decomp}
    v=(I_d-P_0^{j+1})v + L_0 (I_d-P_0^{j+1})v, \quad \forall v \in V.
  \end{equation}
  The characterization \eqref{eq2:charmin} of $\theta_{\min}$ and \eqref{eq3:decomp}
  then show for all $v \in V$
  \begin{equation*}
    \sin(\theta_{\min}) \|v\| \le \dist(v,\range(P_0^{j+1})) \le \|(I_d-P_0^{j+1})v\|, 
  \end{equation*}
  and
  \begin{equation} \label{eq3:estL0}
    \| L_0(I_d-P_0^{j+1})v\| \le \|v\|+\|(I_d-P_0^{j+1})v\|  \le (1+\frac{1}{\sin(\theta_{\min})})
    \|(I_d-P_0^{j+1})v\|.
    \end{equation}
  We iterate the decomposition \eqref{eq3:decomp}
  \begin{equation} \label{eq:ndecomp}
    \begin{aligned}
      \Phi(n,0)v& = \Phi(n,0)(I_d-P_0^{j+1})v\\
      &+\Phi(n,0)P_0^{j+1}L_0(I_d-P_0^{j+1}) \Phi(0,n)(I_d-P_n^{j+1}) \Phi(n,0)
      (I_d-P_0^{j+1})v.
    \end{aligned}
    \end{equation}
 Let $w=\Phi(n,0) (I_d-P_0^{j+1})v\in \kernel(P_n^{j+1})$. Using \eqref{eq3:estL0} and the dichotomy estimates
  \eqref{eq3:dichprop} for $\Phi(n,0)P_0^{j+1}$  and $\Phi(0,n)(I_d-P_n^{j+1})$, we obtain
  \begin{equation} \label{eq3:estPhinvw}
    \|\Phi(n,0)v-w\| \le K_{j}(\eps)^2\left(\frac{\sigma^+_{j+1}+\varepsilon}{\sigma_j^- -\varepsilon}
    \right)^n\left(1+\frac{1}{\sin(\theta_{\min})}\right)\|w\|.
    \end{equation}
    Further, by the invariance of fibers we have
  \begin{equation*}
    \Phi(n,0)\kernel(P_0^{j+1})=\Phi(n,0) \textstyle \medoplus_{i=1}^j \cW_0^i=\medoplus_{i=1}^j \cW_n^i
    =\kernel(P_n^{j+1}).
  \end{equation*}
  For every
  $w \in\kernel(P_n^{j+1})$ there exists a unique element $v \in V$ with $w=\Phi(n,0)(I_d-P_0^{j+1})v$.
  Thus the characterization \eqref{eq2:chardist} and \eqref{eq3:estPhinvw} lead to
  \begin{equation*}
    \begin{aligned}
      \sin( \ang(\kernel(P_n^{j+1}),\Phi(n,0)V)) &= \sup\{\dist(w,\Phi(n,0)V): w \in \kernel(P_n^{j+1}), \|w\|=1\}\\
        & \le K_{j}(\eps)^2\left(\frac{\sigma^+_{j+1}+\varepsilon}{\sigma_j^- -\varepsilon}
        \right)^n\left(1+\frac{1}{\sin(\theta_{\min})}\right).
    \end{aligned}
  \end{equation*}
  Using $\frac{2}{\pi}x \le \sin(x)\le 1$ for $0 \le x \le \frac{\pi}{2}$ and setting
  $\varepsilon=\tfrac{1}{4}(\sigma_j^- - \sigma_{j+1}^+)$, we obtain the first estimate in \eqref{eq3:alphanest} for the constants
  \begin{equation*}
    r= \tfrac{1}{2} \Big(1 + \frac{\sigma_{j+1}^+}{\sigma_j^-}\Big),
    \quad C_0= \tfrac{1}{2}\pi^2 K_{j}(\eps)^2.
      \end{equation*}
   With the  triangle inequality 
  for ``$\ang$'' we then find
  \begin{equation*}
    \begin{aligned}
      & |\alpha^{\disc}(n,V)- \alpha^{\disc}(n, \textstyle \medoplus_{i=1}^j \cW_0^i)|\\
      &\le \frac{1}{n} \sum_{k=1}^n
      \left|    \ang(\Phi(k-1,0)V,\Phi(k,0)V) 
      -\ang(\Phi(k-1,0)\cT_0(V),\Phi(k,0)\cT_0(V)) \right| \\
      & \le \frac{1}{n} \sum_{k=1}^n\    \ang(\Phi(k-1,0)V,\Phi(k-1,0)\cT_0(V)) + 
      \ang(\Phi(k,0)V,\Phi(k,0)\cT_0(V)) \\
      & \le \frac{2C_0}{n\theta_{\min}} \sum_{k=0}^n r^k \le \frac{2 C_0}{n(1-r)\theta_{\min}}.
    \end{aligned}
  \end{equation*}
  Finally, the equality of angular ranges follows from \eqref{eq3:angrangeequal} and \eqref{eq3:fspaces}.
\end{proof}

In the particular case of one-dimensional fibers we can choose $j=s$ in \eqref{eq3:sexact}. Since
a random choice of $V$ almost always satisfies the second condition in \eqref{eq3:sexact},
we find that Algorithm \ref{alg2} almost surely yields the angular range of the subspace spanned by
the $s$ fastest fibers.

On the other hand, if $s < \sum_{i=1}^j d_i$ holds in \eqref{eq3:determinej}, then the reduced set 
of trace spaces
\begin{equation} \label{eq3:redtrace} \textstyle
    \cD_0^j(s,d)= \Big\{ \medoplus_{i=1}^{j} W^i\in \cD_0(s,d):  \sum_{i=1}^{j} \dim W^i= s \Big\}
    \end{equation}
   is a continuum. Therefore,  we must invoke some continuous optimization method for 
finding the infimum and the supremum of  the angular spectrum.

In Section \ref{sec1.3} we used the inverse map to find the angular spectrum associated
to the sum of the $s$ slowest fibers. We derive  a suitable analog of
 Proposition \ref{cor3:anglefast}.
Since we cannot invert from infinity, we invert the solution operator at any finite $N \in \N$
and take care that the resulting estimates are independent of $N$. The inverse solution operator
is 
\begin{equation*}
  \widetilde{\Phi}(n,m) = \Phi(N-n,N-m), \quad 0 \le n,m \le N.
\end{equation*}
Note that we do not make the dependence on $N$ explicit.
It is not difficult to see that the properties \eqref{eq3:projprop}, \eqref{eq3:dichprop}
and the decomposition \eqref{eq3:fibers}
hold for $\widetilde{\Phi}$ with constants independent of $N$ and the
following data ($n=0,\ldots,N$):
\begin{equation*} \label{eq:invdata}
\begin{aligned}
  \tilde{\sigma}_i^{\pm}= \left(\sigma_{\varkappa-i+1}^{\mp}\right)^{-1},
  \quad \widetilde{\cW}_n^i = \cW_{N-n}^{\varkappa-i+1}, \quad i=1,\ldots,\varkappa, \\
 \quad \tilde{P}_n^i= I_d - P_{N-n}^{\varkappa -i+2}, \quad i=1,\ldots,\varkappa+1.
\end{aligned}
\end{equation*}
Given a subspace $\widetilde{V} \in \cG(s,d)$ at time $N$, the trace spaces cannot be computed in
terms of the forward iteration, but are defined as follows
\begin{equation*}
  \widetilde{\cT}_{N-n}(\widetilde{V})= \textstyle \medoplus_{i=1}^{\varkappa} P_{N-n}^{\varkappa-i+1}\left(\kernel(P_{N-n}^{\varkappa-i+2})\cap \widetilde{V}\right), \quad n=0,\ldots,N.
\end{equation*}

Using the dichotomy estimates for the reversed system, one can show the following result.

\begin{proposition} \label{cor3:angleslow}
Let $0<s<d$. There exist constants $0<r<1$ and $C>0$,
depending only on the dichotomy estimates, such that the following
holds.
For every $N\in\N$, $j\in\{1,\ldots,\varkappa\}$,
and $\widetilde V\in\cG(s,d)$ satisfying
\begin{equation}\label{eq3:sinvert}
 s=\sum_{i=j}^{\varkappa}d_i,
 \qquad
 \widetilde V\cap
 \textstyle\medoplus_{i=1}^{j-1}\cW_N^i=\{0\},
\end{equation}
define $\tilde{\theta}_{\min}=\theta_{\min}( \widetilde V, \medoplus_{i=1}^{j-1}\cW_N^i)$.
Then the estimates hold
\begin{equation}\label{eq3:alphainvert}
  \begin{aligned}
 \ang\big(
   \Phi(N-n,N)\widetilde V,\,
   \textstyle\medoplus_{i=j}^{\varkappa}\cW_{N-n}^i
 \big)
 &\le \frac{C}{\tilde{\theta}_{\min}} r^n,
 \quad 0\le n\le N, \\
 \big|
 \tilde{\alpha}_s^{\disc}(N-n,\tilde V)-
 \tilde{\alpha}_s^{\disc}\big(
 N-n,\textstyle\medoplus_{i=j}^{\varkappa}\cW_N^i
 \big)
 \big| & \le
 \frac{C }{n \tilde{\theta}_{\min}},
 \quad 1\le n\le N,
\end{aligned}
\end{equation}
where, for $V\in\cG(s,d)$,
\[
 \tilde{\alpha}_s^{\disc}(N-n, V)
 :=
 \frac{1}{n}\sum_{k=1}^{n}
 \ang\left(
 \Phi(N-k+1,N)V,\,
 \Phi(N-k,N) V
 \right).
\]
\end{proposition}

Our numerical computations in Section \ref{sec1.3} yield the quantity $\tilde{\alpha}_s^{\disc}(0,\widetilde{V})$, which according to \eqref{eq3:alphainvert} is a good approximation of
  \begin{equation*}
    \tilde{\alpha}_s^{\disc}(0, \textstyle \medoplus_{i=j}^{\varkappa}\cW_{N}^i)=\frac{1}{N} \sum_{\ell=1}^N
    \ang(\textstyle \medoplus_{i=j}^{\varkappa}\cW_{\ell}^i,\textstyle \medoplus_{i=j}^{\varkappa}\cW_{\ell-1}^i),
  \end{equation*}
  i.e., of the angular spectrum that belongs to the slowest fibers.

  \subsection{Combining forward and backward iteration}
  \label{sec3.1a}
Our numerical computations indicate that, for the trajectories of the
H\'enon models considered here, all spectral fibers are one-dimensional. 
Consequently, for the two-dimensional
system $f_2$, we compute the entire angular spectrum by applying Algorithm~\ref{alg1b}
to $f_2$ and $f_2^{-1}$. 

For the three-dimensional H\'enon system $f_3$, however, this approach
yields only the angular value of the fastest trace space $\cW_n^1$
(via forward iteration) and the slowest trace space $\cW_n^3$ (via backward
iteration). 
Furthermore, using Algorithm~\ref{alg2adj}, we have already computed the
normal vector $\eta_n^{1,2}$, which is orthogonal to $\cW_n^1 \oplus
\cW_n^2$ via forward iteration as well as the normal vector
$\eta_n^{2,3}$ via backward iteration, which is orthogonal to $\cW_n^2 \oplus
\cW_n^3$. The cross product $\omega_n^2=\eta_n^{1,2} \times \eta_n^{2,3}$ of these two
vectors spans the second trace space, i.e., $\Span(\omega_n^2) =
\cW_n^2$ for $n\in\{0,\dots,N\}$.
 
In summary, for $i\neq j\in\{1,2,3\}$ and
$n\in\{0,\dots,N\}$, the algorithms described above yield vectors $\omega_n^i$ and $\eta_n^{i,j} =
\omega_n^i \times \omega_n^j$ such that 
\[
\Span (\omega_n^i) = \cW_n^i \quad \text{and}\quad 
\Span (\eta_n^{i,j}) = (\cW_n^i \oplus \cW_n^j) ^\bot.
\]  
Using these subspaces together with \eqref{eq:perpequal}, we compute
the approximate angular spectra $\Sigma_{s,N}^{\disc}$ without requiring any further iterations:
\begin{align*}
\Sigma_{1,N}^{\disc} &= \{\tilde\alpha_1^{\disc}(0,\cW^i_N): i\in\{1,2,3\}\}
= \left\{\textstyle \frac 1 N \sum_{\ell=1}^N
                   \ang(\omega^i_\ell,\omega_{\ell-1}^i):
                   i\in\{1,2,3\}\right\},\\  
\Sigma_{2,N}^{\disc} &= \{\tilde\alpha_2^{\disc}(0,\cW^i_N\oplus \cW^j_N): i,j\in\{1,2,3\},\
                   i< j\}\\
&=  \left\{\textstyle \frac 1 N \sum_{\ell=1}^N
                   \ang(\eta^{i,j}_\ell,\eta_{\ell-1}^{i,j}):
                   i,j\in\{1,2,3\},\ 
                   i<j\right\}.
\end{align*} 

Figure \ref{Fh6} shows the resulting histograms.
In \cite[Section 4.2.4]{BeHu22}, the numerical results for both H\'enon
models were obtained using the general algorithm described in
Section \ref{sec3.2}. Here, we reproduce these results using forward
and backward iterations without having to compute the dichotomy spectrum. 

\begin{figure}[H]
\begin{center}
\begin{tabular}{cc}
$s = 1$ & $s=2$\\[-1mm]
\includegraphics[width=0.47\textwidth]{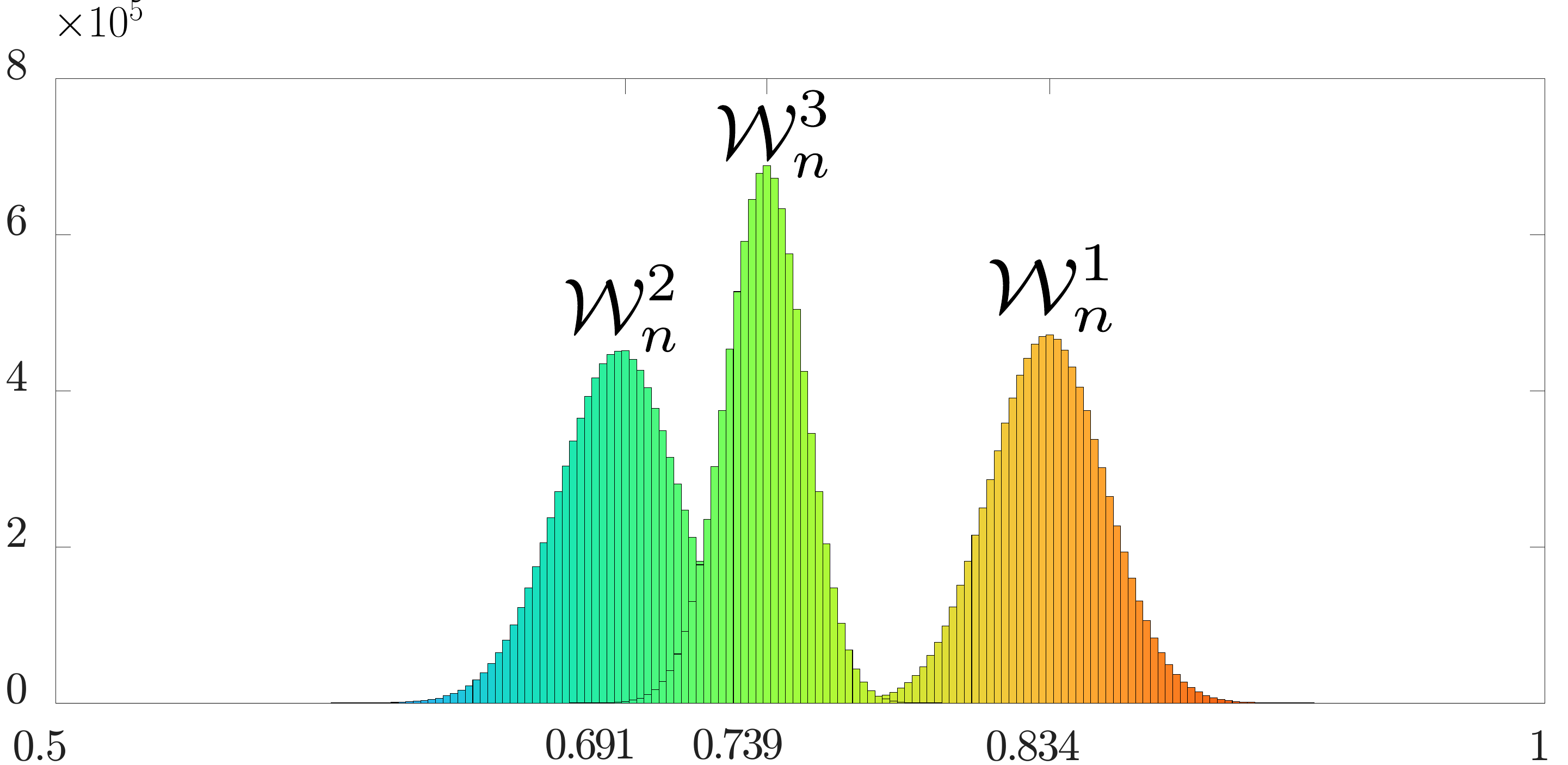}&
\includegraphics[width=0.47\textwidth]{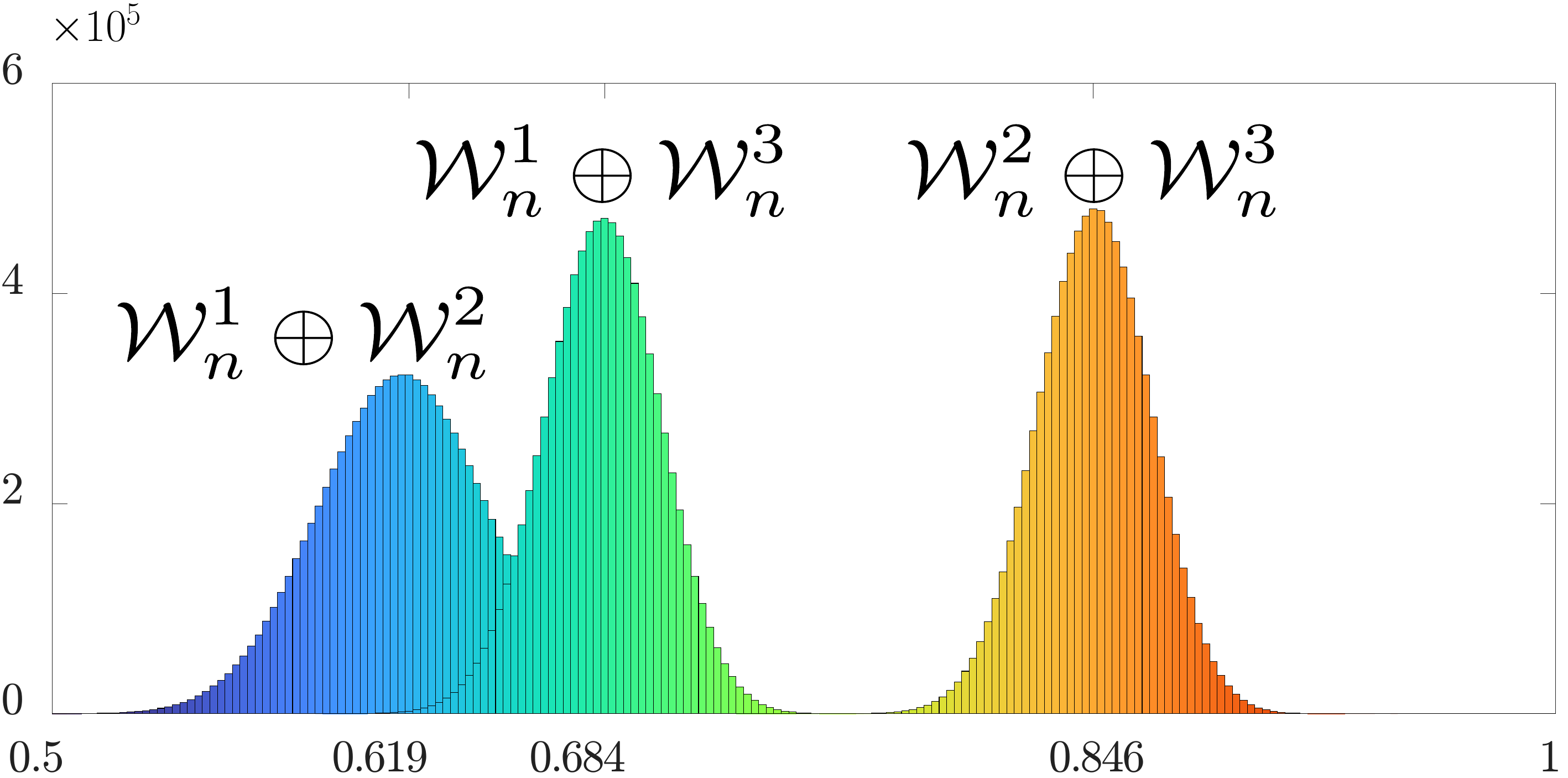}
\end{tabular}      
\end{center}
\caption{Histograms of the angular map for the three-dimensional
  H\'enon map $f_3$, computed with  $N=1000$, for $s=1$ (left) and
  $s=2$ (right). Each peak is  labeled with its associated trace space.  \label{Fh6}}   
\end{figure}

In Figure \ref{Fh7}, we vary the length $N$ of the computed orbit over
the range $10^2\leq N\leq 10^7$ while keeping the initial point fixed
at $x_0=(0.2,\ 0.1,\ 0)^\top$. The orbit starting at $x_0$ approaches the
attractor of the three-dimensional H\'enon system. 
Our numerical results suggest that the six angular sums converge to angular spectral values as
$N$ increases. 
  
\begin{figure}[hbt]
\begin{center}
\includegraphics[width=0.7\textwidth]{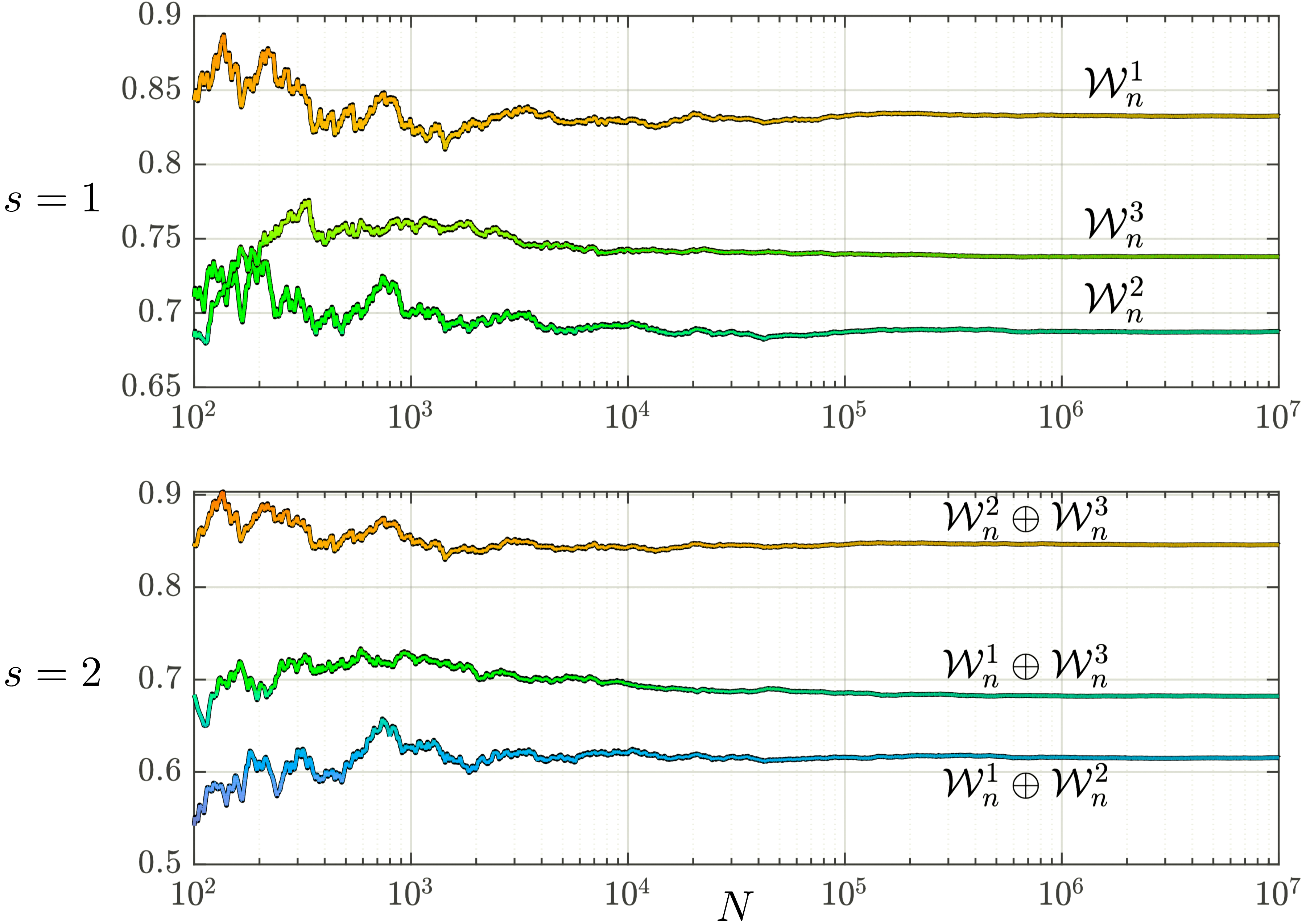}
\end{center}
\caption{Angular values for the fixed initial point $x_0=(0.2,\ 0.1,\ 0)^\top$ and
orbit lengths ranging from $N=10^2$ to $N=10^7$, with $s=1$ in the
upper panel and $s=2$ in the lower panel. The color coding is the same
as in Figure \ref{Fh6}. \label{Fh7}}  
\end{figure}

\subsection{The general algorithm in discrete time}
\label{sec3.2}
In Section \ref{sec3.1a} we approximated the entire angular spectrum along the
trajectories considered for the three-dimensional Hénon system using
the QR-like Algorithm \ref{alg2}. This worked since
all fibers are one-dimensional and since we employed techniques like inversion of the
system and forming orthogonal complements or intersections of fibers. 
For systems of dimension $d>3$, combining forward and backward
iteration becomes more involved. For these systems, as well as
for systems with spectral fibers of dimension $\ge 2$, we use
the trace-space approach described below to obtain a complete
picture of the angular spectrum and the corresponding angular map.
 
According to  Section \ref{sec3.1} we
can achieve this by computing the angular ranges \eqref{eq2:angrange} for all trace
spaces $V \in \cD_0(s,d)$; see \eqref{trace}. This in turn requires to first compute the
dichotomy spectrum and its associated invariant foliation. Then in a second step
one can form the possible trace spaces and compute their angular spectral values. Such an
algorithm was set up  in \cite[Section 4.2]{behu24b}
under the assumption that all spectral fibers are one- or two-dimensional.
The algorithm proves to be
quite expensive if the set of trace spaces is not finite but consists of one- or  two-dimensional
continua (see the remarks following \eqref{eq3:redtrace}).
In such cases, we limit ourselves to computing  lower and upper
bounds of the set $\Sigma_{s,N}^{\disc}(x_0) = \{\alpha_s^{\mathrm{disc}}(N,V;x_0):
V\in\mathcal D_0(s,d;x_0)\}$, given by
\begin{equation*}
  \Sigma_{s,N}^{\disc, \min}(x_0) =\inf \ \Sigma_{s,N}^{\disc}(x_0), \quad
  \Sigma_{s,N}^{\disc, \max}(x_0) =\sup \ \Sigma_{s,N}^{\disc}(x_0).
  \end{equation*}
Here $\mathcal D_0(s,d;x_0)$ denotes the set of trace spaces
associated with the variational equation along the trajectory
starting at $x_0$.

We consider maps    
$F:\R^d \to \R^d$ and let  $\cB= \bigtimes_{j=1}^d[a_j,b_j]$ be a
$d$-dimensional  rectangular box. By $L$ we denote  the resolution of our grid and 
define $\cL =\{1,\ldots,L\}$ and $h_i = \frac {b_i-a_i}L$ for $i=1,\dots,d$. 
For every multi-index
\[
    \bi=(i_1,\ldots,i_d)\in\mathcal L^d,
\] 
we denote the corresponding grid box by
\[
    \mathcal B_{\boldsymbol{i}}
    :=\bigtimes_{j=1}^d [a_j+(\boldsymbol{i}_j-1)h_j, a_j+ \boldsymbol{i}_jh_j].
\] 
 
The general algorithm reads:
\begin{algorithm}
\caption{Angular maps for $s \in \{1,2\}$ and general $d$ \label{alg4}}
\textcolor{header1}{
\begin{algorithmic}[1] 
\For {$\bi\in\cL^d$}
\State Choose the midpoint $x_0 \in \cB_{\bi}$
\For {$n = 1,\dots,N$}
\State $x_{n} = F_{n-1}(x_{n-1})$
\State $A_{n-1} = DF_{n-1}(x_{n-1})$
\EndFor
\State Compute the angular spectral values
$\Sigma_{s,N}^{\disc,\min}$ and $\Sigma_{s,N}^{\disc,\max}$
for the linear difference equation
$u_{n+1}=A_nu_n$, $n=0,\dots,N-1$,
using the algorithm from \cite[Section 4.2]{behu24b}
with buffer intervals of length $B$
\State Color the box $\cB_{\bi}$  according to
$\Sigma_{s,N}^{\disc,\min}$ and $\Sigma_{s,N}^{\disc,\max}$, respectively
\EndFor
\end{algorithmic}}
\end{algorithm}

The role of the buffer intervals is explained below. 
In Algorithm \ref{alg4}, we color the initial box of a starting point according to the minimal resp.\ the maximal angular value of the
corresponding trajectory. These
values depend on the finite length $N$ of  the trajectory, but we expect them to converge as $N \to \infty$.
 
The major work of the algorithm occurs in step 6, where the
angular spectral values
$\Sigma_{s,N}^{\disc, \min}$ and $\Sigma_{s,N}^{\disc, \max}$ are computed. It is essential
that the algorithm in \cite[Section 4.2]{behu24b} computes these values 
for the trace spaces $V$ in $\cG(s,d)$ rather than for all
elements of $\cG(s,d)$; see \eqref{trace}.
The computation of trace spaces requires most of the computational effort. In some
cases, such as the Lorenz system treated in Section \ref{sec_lorenz}, it turns out that 
only finitely many trace spaces (in fact, very few) are analyzed by the algorithm. 

Approximation errors occurring in the computation of trace spaces
decay exponentially fast from the boundary of the
computational interval toward its center; see \cite{h17}, \cite[Section~4.2]{behu24b}. Thus, we introduce left and
right buffer intervals and utilize only the accurate data between these buffer
intervals for computing the angular maps.   

\subsection{The  area-preserving H\'enon map}
\label{sec3.4}
We consider an area-preserving version of H\'enon's map  (see
\cite{MCC20})
\begin{equation}\label{henonap}
f_{a}:\begin{array}{rcl} \R^2 & \to & \R^2\\
\begin{pmatrix}x_1\\x_2\end{pmatrix} & \mapsto &
\begin{pmatrix} a -x_1^2 + x_2\\ x_1\end{pmatrix},
\end{array}
\end{equation}
with parameter $a=0.3$.
As we show below, trajectories on certain invariant curves surrounding
the elliptic periodic points have dichotomy spectrum
$\Sigma_{\mathrm{ED}}=\{1\}$. For these trajectories, the only
spectral fiber is $\R^2$, so determining the extremal angular
spectral values requires optimization over all one-dimensional
subspaces. 

For a numerical experiment we choose $\cB=[-1,1]^2$. Figure \ref{Fh8} shows the midpoints of an $L\times
L$ grid, with $L=10^3$, whose trajectories remain in $\cB$ throughout
$N=2000$ iterations. 

 For every grid point, the left panel
displays the computed minima, whereas the right
panel displays the computed maxima. The
corresponding histograms are shown in the lower row. Figure 
\ref{Fh9} displays the pointwise difference between the maximal and
minimal angular spectral values shown in Figure \ref{Fh8}.
A logarithmic color scale is used to make small differences more visible. 

\begin{figure}[hbt]
\begin{center}
\includegraphics[width=0.95\textwidth]{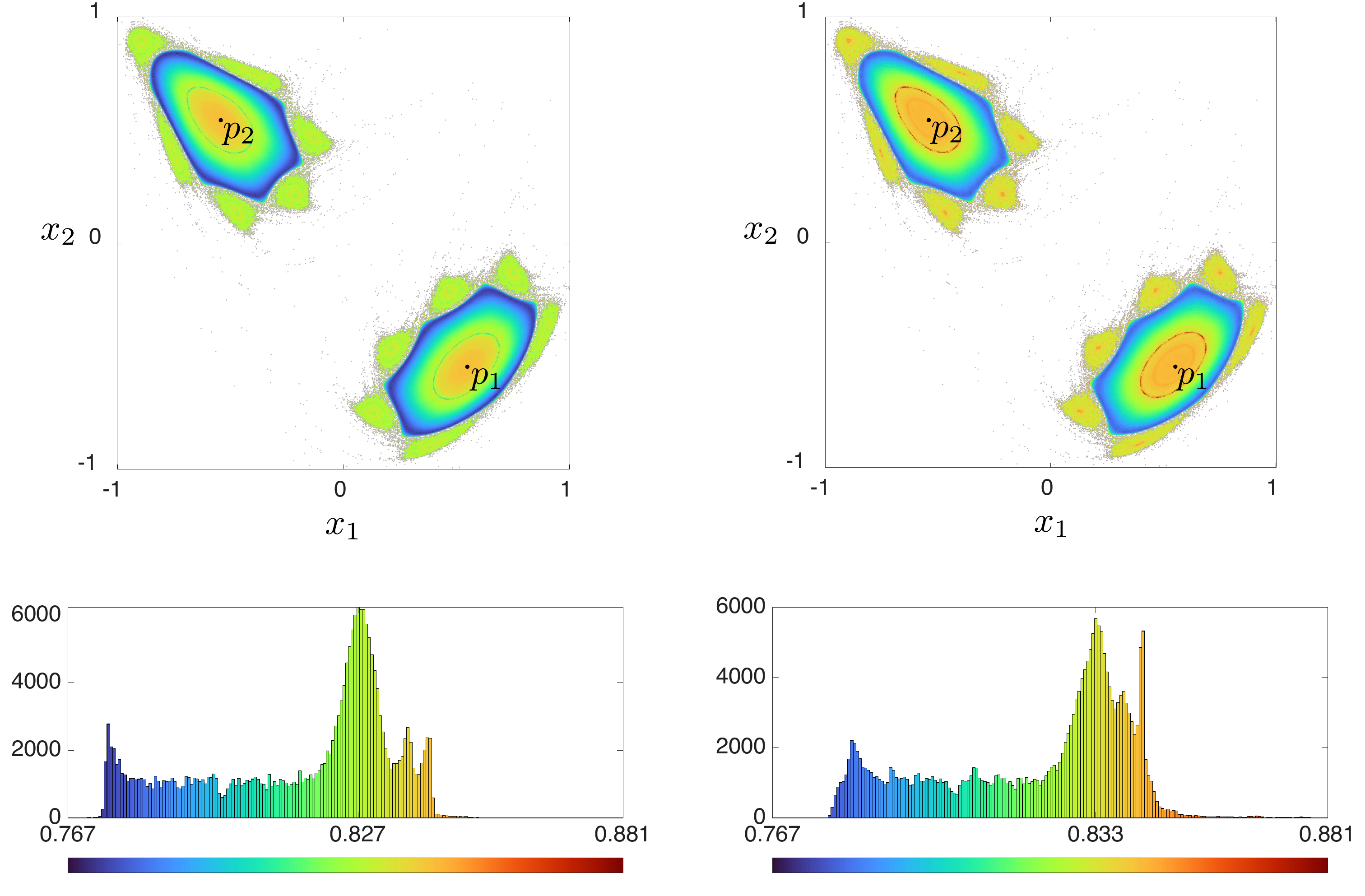}
\end{center}
\caption{Angular maps of the minimum (left) and maximum (right)
  angular spectral values for the area-preserving H\'enon
  map \eqref{henonap}. The corresponding histograms are shown in the
  lower row. \label{Fh8}}  
\end{figure}

\begin{figure}[hbt]
\begin{center}
\includegraphics[width=0.45\textwidth]{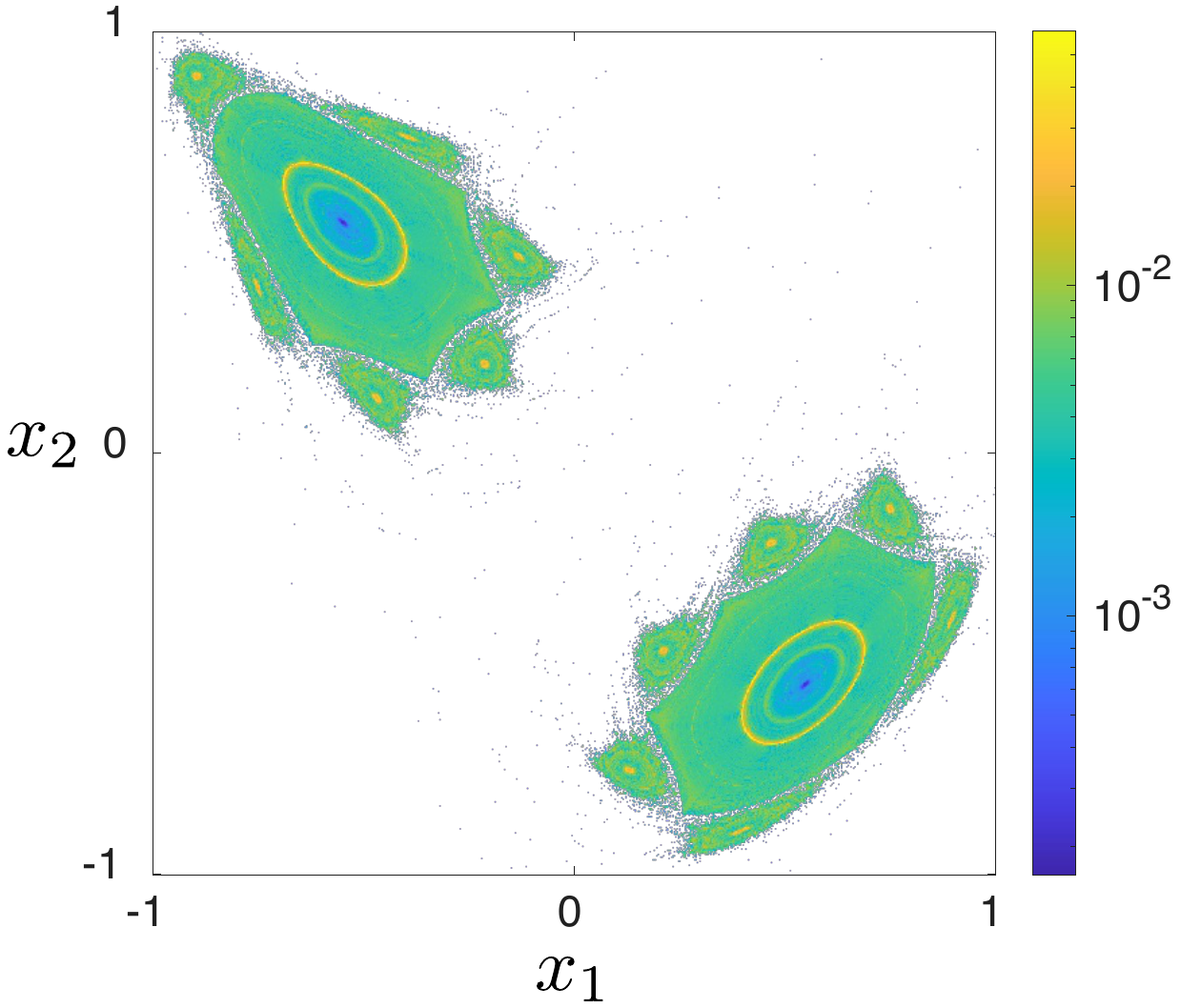}
\end{center}
\caption{Pointwise difference between the maximal and minimal angular
  spectral values shown in Figure \ref{Fh8}. The colors are
  displayed on a logarithmic scale. \label{Fh9}}  
\end{figure}
Let us add a few analytical findings for this system with $0<a<1$, which are in accordance with the
numerical observations. Area preservation follows from $|\det(Df_{a}(x))|=1$ for all
$x \in \R^2$. Then  there are two hyperbolic fixed points $x_{\pm}=(\pm\sqrt{a},\pm\sqrt{a})^{\top}$
for which $Df_a(x_{\pm})$ has eigenvalues of modulus
 $\sqrt{a}+\sqrt{1+a}>1$ and $ -\sqrt{a}+\sqrt{1+a}<1$.
Further, there is a $2$-periodic orbit
\begin{equation*}
  p_1=\begin{pmatrix} \sqrt{a} \\ -\sqrt{a} \end{pmatrix}, p_2=f_a(p_1)=-p_1
\end{equation*}
located at the centers of the two colored domains in Figure \ref{Fh8}. The points $p_{1,2}$
are elliptic fixed points of $f^2$, since $D(f_a^2)(p_1)=Df_a(p_2)Df_a(p_1)$ has eigenvalues
\begin{equation*}
  \lambda_{+}= 1 -2a + 2 i \sqrt{a(1-a)}, \quad \lambda_-=\overline{\lambda_{+}},
    \quad |\lambda_{\pm}|=1.
\end{equation*}
Therefore, the linearization about the two-periodic orbit $p_1,p_2,p_1,p_2,\ldots$ has no
exponential dichotomy.
One  can say much more about trajectories starting near these elliptic points
by applying Moser's Twist Theorem \cite[Ch.\ II.4]{Moser1973}.
\begin{proposition} \label{prop3:twist}
  Let $a \in (0,1)\setminus \{\tfrac{1}{2}, \tfrac{3}{4},
  \tfrac{5}{8}\}$. In every neighborhood of $p_1$ there are invariant
  curves of $f_a^2$,  
  on which $f_a^2$ is $C^1$-conjugate to a rigid rotation with some diophantine frequency
  $\omega$, i.e., $|\tfrac{\omega}{2 \pi}-\tfrac{p}{q}|\ge C
  q^{-\ell}$ for all $p \in \Z$, $q \in \N$ and some $\ell>2$,
  $C>0$.  Initial points on these curves lead to 
  trajectories for which the linearization of $f_a$ has dichotomy spectrum $\Sigma_{\ED}=\{1\}$.
\end{proposition}  
\begin{proof}
  First note that the assumption $a\in (0,1)\setminus \{\tfrac{1}{2}, \tfrac{3}{4}\}$
    excludes the $1:k$ resonance $\lambda_+^k=1$ for $k=1,\ldots,4$. Then there
    are local coordinate changes under which $f_a^2$ near $p_1$ can be put into
    Birkhoff normal form \cite[Ch.\ II.4, Theorem 2.12]{Moser1973}, 
   given in complex notation as 
  \begin{equation*} \label{eq3:Birkhoffnormalform}
    \C \ni z \mapsto \lambda_+z(1+i B_1(a) |z|^2) + \mathcal{O}(|z|^4).
  \end{equation*}
  Note that the variable $z$ can still be scaled by a nonzero complex number, so that the 
  first Birkhoff coefficient $B_1(a)$ is only unique up to a positive multiple.
  With a suitable normalization, a symbolic toolbox yields the expression
  \begin{equation*}
    B_1(a)= \frac{8a-5}{4 \sqrt{a}(1-a)(4a-3)}.
  \end{equation*}
  Hence, we have $B_1(a)\neq 0$ for $a \neq \tfrac{5}{8}$, and the Twist Theorem 
  \cite[Ch.\ II.4, Theorem 2.11]{Moser1973} applies.
  Let  $\Gamma$ be one of the emerging invariant curves. Then there exists a $C^1$-immersion  $H: S^{2\pi}=\R/(2 \pi \Z) \to \R^2$  such that $H(S^{2 \pi})=\Gamma$ and
    \begin{equation} \label{eq3:f2conjugate}
      f_a^2 \circ H = H \circ R_{\omega}, \quad R_{\omega}(\theta)= \theta+ \omega \; \text{mod}
     \;  2 \pi,
    \end{equation}
    for some $\omega$ satisfying the diophantine condition.

    For $\theta \in S^{2 \pi}$ the vectors
    $\tau(\theta)=H'(\theta)$ and
    $\rho(\theta)=J \tau(\theta)\|\tau(\theta)\|^{-2}$, $J=\left(\begin{smallmatrix} 0 & -1 \\1 &0\end{smallmatrix}\right)$,
      form an orthogonal system at $H(\theta)$ satisfying
      $\det\left( \begin{smallmatrix} \tau(\theta) & \rho(\theta) \end{smallmatrix}\right)=1$.
        By differentiating
    \eqref{eq3:f2conjugate} one finds
    \begin{equation} \label{eq3:df2rep}
      D(f_a^2)(H(\theta))\begin{pmatrix} \tau(\theta) & \rho(\theta) \end{pmatrix}=
      \begin{pmatrix} \tau(\theta+\omega) & \rho(\theta+\omega) \end{pmatrix}
      \begin{pmatrix} 1 & c(\theta) \\ 0 & d(\theta) \end{pmatrix}.
    \end{equation}
    Taking the determinant, we obtain $d(\theta)\equiv 1$ and 
    \begin{equation*}
      c(\theta)= \frac{ \tau(\theta+\omega)^{\top} D(f_a^2)(H(\theta)) \rho(\theta)}{\| \tau(\theta+\omega)\|^2}.
    \end{equation*}
    By induction the representation of the Jacobian in suitable coordinates \eqref{eq3:df2rep} leads to

    \begin{equation} \label{eq3:df2it}
      D(f_a^{2n})(H(\theta))\begin{pmatrix} \tau(\theta) & \rho(\theta) \end{pmatrix}=
      \begin{pmatrix} \tau(\theta+n \omega) & \rho(\theta+n \omega) \end{pmatrix}
      \begin{pmatrix} 1 & \sum_{j=0}^{n-1}c(\theta+j \omega) \\ 0 & 1 \end{pmatrix}.
    \end{equation}
    Since $c(\theta)$ and the matrix $\left(\begin{smallmatrix} \tau(\theta) & \rho(\theta) \end{smallmatrix}
    \right)$ as well as its inverse are uniformly bounded, we obtain a constant $C>0$ such
    that
    \begin{equation*}
      \|D(f_a^{2n})(H(\theta))\| \le C(1+ |n|) \quad \text{for all} \; n \in \Z, \; \theta \in S^{2\pi}.
    \end{equation*}
    Thus, the variational equation for $f_a^2$ has no exponentially growing solutions in forward or backward
    time. If $f_a$ has such a solution, it will have the same behavior
    for $f_a^2$, and therefore we 
    conclude $\Sigma_{\ED}=\{1\}$.
      \end {proof}

\begin{remark}\label{rem3:polynomial}
In \eqref{eq3:df2it}, the Jacobians $D(f_a^{2n})$ are represented
in local coordinates by a shear matrix
$\left( \begin{smallmatrix} 1 & c_n \\
    0 & 1 \end{smallmatrix}\right)$, where
$c_n=\sum_{j=0}^{n-1}c(\theta+j\omega)$.
On sufficiently small Moser invariant curves, the nonzero twist
implies linear growth of these shear sums. Consequently,
nontangential directions approach the tangent direction at a
rate $\cO(n^{-1})$.
In numerical computations, we also observe that one-dimensional
subspaces approach each other at a polynomial rate.
However, we do not exploit this behavior and retain the original
optimization procedure.
\end{remark}


\section{Angular maps for continuous flows and a general algorithm}
\label{sec4} 

In this section we consider general nonautonomous systems in continuous time
\begin{equation} \label{eq3:contnonlin}
  \dot{x}(t) = f(t,x(t)), \quad t \in \R, \quad x(t) \in \R^d,
\end{equation}
defined by some  vector field $f\in C^1(\R\times \R^d,\R^d)$. Extending the theory from
\cite{behu24a}, we set up proper notions of angular spectrum and angular map
for systems obtained by linearizing \eqref{eq3:contnonlin} at a solution.
Further, we show that the angular spectrum of the discretized linear system 
converges to its continuous counterpart when the step-size goes to zero.
\subsection{Definition of angular spectrum and angular maps}
\label{sec4.1}
Let $\Psi(t_1,t_0;x_0)$ denote the solution operator of \eqref{eq3:contnonlin}, which maps the
initial value $x(t_0)=x_0$ to the solution of \eqref{eq3:contnonlin} at time $t_1$.
We choose $t_0=0$ and note that $\Psi(t,0;x_0)$ exists on an open
maximal interval $J=J(x_0) \subseteq \R$ 
containing $0$. The associated variational equation is 
\begin{equation} \label{eq3:contvari}
  \dot{u}(t)= A(t;x_0)u(t),  \; t \in J, \quad A(t;x_0) = \frac{\partial f}{\partial x}(t,\Psi(t,0;x_0)).
  \end{equation}
By $u(t,s;x_0,u_0)$  we denote the solution of this linear system with initial condition $u(s)=u_0$ for some $s \in J$ and $u_0 \in \R^d$. Since the solution depends linearly on $u_0$, this defines the solution operator $\Phi(\cdot,\cdot;x_0): J(x_0) \times J(x_0) \to \R^{d,d}$ via
\begin{equation*}
  \Phi(t,s;x_0)u_0= u(t,s;x_0,u_0).
\end{equation*}
In \cite[Def.\ 4.3]{behu24a} we defined angular values for linear
time-varying systems of the type \eqref{eq3:contvari} using the
following quantities for $T>0$ and $V \in \cG(s,d)$:
\begin{equation} \label{eq3:defcont}
     \alpha_s^{\cont}(T,V;x_0)= \frac{1}{T} \int_0^T \|(I_d- P_{\Phi(t,0;x_0)V})A(t;x_0) P_{\Phi(t,0;x_0)V}\| dt, 
\end{equation}
where $\|\cdot\|$ is the spectral norm and $P_{\Phi(t,0;x_0)V}$ denotes the orthogonal projector
onto the subspace $\Phi(t,0;x_0)V \in \cG(s,d)$. 
 
In what follows we consider a domain $\cB \subseteq \R^d$ such that the solutions of \eqref{eq3:contnonlin} exist for all $t \ge 0$ and
for all $x_0 \in \cB$. Moreover, we assume that
we have a uniform bound
\begin{equation*}
  \| A(t;x_0)\| \le C, \quad \forall \, t\ge 0, x_0 \in \cB.
\end{equation*}
This implies the bound $\alpha_s^{\cont}(T,V;x_0)\le C$ for the values in \eqref{eq3:defcont} and hence that the limit superior and
limit inferior are finite.
This justifies the following analog of Definition \ref{def2:angspec}.

\begin{definition} \label{def3:contang}
  For any $ V \in \cG(s,d)$ and $x_0 \in \cB$ the interval
  \begin{equation} \label{eq3:contrange}
     I_s^{\cont}(V;x_0) = \left[
       \varliminf_{T\to \infty} \alpha_s^{\cont}(T,V;x_0) , \varlimsup_{T\to \infty}\alpha_s^{\cont}(T,V;x_0) \right]
  \end{equation}
  is called the $s$-dimensional \textbf{angular range} of $V$ for the linear system \eqref{eq3:contvari}. 
  The {\bf angular spectrum} of dimension $s \in \{1,\ldots,d\}$ is defined by
  \begin{equation} \label{eq3:defangspec}
    \Sigma_s^{\cont}(x_0) = \mathrm{cl}\Big[ \bigcup_{V \in \cG(s,d)} I_s^{\cont}(V;x_0) \Big]. 
  \end{equation}
  The map
  \begin{equation*}
\Sigma_s^{\cont}: \begin{array}{rcl} \cB & \to & \mathcal{A}(\R)\\
x_0 & \mapsto & \Sigma_s^{\cont}(x_0)
\end{array}
  \end{equation*}
  is called the \textbf{angular map} of the system
  \eqref{eq3:contvari} in the domain $\cB\subseteq \R^d$. Here
  $\mathcal{A}(\R)$ denotes the collection of nonempty compact subsets of $\R$. 
\end{definition}
The outer angular values defined in \cite[Def.\ 4.3]{behu24a} depend on $x_0$ in our setting. They are related to the angular spectrum as follows:
\begin{equation*}
  \vartheta_s^{\inf,\varliminf}(x_0)=\inf \ \Sigma_s^{\cont}(x_0), \quad \vartheta_s^{\sup,\varlimsup}(x_0)= \sup \  \Sigma_s^{\cont}(x_0).
  \end{equation*}
The analog of Lemma \ref{lem2:accu} in the continuous case is the following.
  \begin{lemma} \label{lem4:accucont}
    The value  $\theta \in\R$ belongs to  $I_s^{\cont}(V;x_0)$ if and only if there exists a sequence $T_n \to \infty$ such that 
    $\lim_{n \to \infty} \alpha_s^{\cont}(T_n,V;x_0) = \theta$.
  \end{lemma}
  \begin{proof} If $\theta$ is one of the endpoints of $I_s^{\cont}(V;x_0)$, then the assertion is clear. Otherwise,
    we find for every $n \in \N$ some $T_-(n)\ge n$ with $\alpha_s^{\cont}(T_-(n),V;x_0) < \theta$.
    And there exist numbers $T_+(n) \ge T_-(n)$ with $\alpha_s^{\cont}(T_+(n),V;x_0) > \theta$. The function $\alpha_s^{\cont}(T,V;x_0)$
    is continuous w.r.t.\ $T$; hence, there exist $T_n \in [T_-(n),T_+(n)]$ with $\alpha_s^{\cont}(T_n,V;x_0) = \theta$ by the intermediate
    value theorem.
    Since $T_n \to \infty$, this proves our assertion.
  \end{proof}

The integral expression in \eqref{eq3:defcont} was obtained  
in \cite[Section 4]{behu24a} as the limit
\begin{equation*} \alpha_s^{\cont}(T,V;x_0)=  \lim_{h \to 0} \alpha_s^{\step}(h,T,V;x_0)
\end{equation*}
of angular averages  $\alpha_s^{\step}$ which are formed with a stepsize $h=\frac{T}{N}$ as follows:
\begin{equation} \label{eq3:limtheta}
      \alpha_s^{\step}(h,T,V;x_0)  =   \frac{1}{T} \sum_{j=1}^N \ang(\Phi((j-1)h,0;x_0)V,\Phi(jh,0;x_0)V), \quad h=\frac{T}{N}.
\end{equation}
Note that the term $\frac{1}{T}= \frac{1}{hN}$ in \eqref{eq3:limtheta} carries an additional factor $\frac{1}{h}$ when compared to
\eqref{def:alpha}. The reason is that the quantity to be averaged over time is the rate of change of the principal angles per unit time, rather than
the principal angles themselves. In \eqref{eq3:limtheta} the time unit is $h$, whereas 
the discrete formula \eqref{def:alpha} tacitly assumes that one step corresponds to a time unit of length $1$.
In Section \ref{sec4.4} below we provide a precise estimate of the error $|\alpha_s^{\cont}(T,V;x_0)-\alpha_s^{\step}(h,T,V;x_0)|$.
This will again justify that the sum in \eqref{eq3:limtheta} is scaled correctly.

\subsection{The general algorithm for flows}\label{sec4.2}
Let $T>0$ be the length of the time interval and let $h = \frac T N$
be the step size. We use the multi-index notation from Section
\ref{sec3.2}.
\begin{algorithm}
\caption{Angular maps for $s \in \{1,2\}$ and general $d$ \label{alg6}}
\textcolor{header1}{
\begin{algorithmic}[1] 
\For {$\bi\in\cL^d$}
\State Choose the midpoint $x_0 \in \cB_{\bi}$
\For {$n = 1,\dots,N$}
\State $x_{n} = RK(f,h,x_{n-1})$
\State $A_{n-1} = \frac {\partial RK}{\partial x}(f,h,x_{n-1})$
\EndFor
\State Compute the angular spectral values
$\Sigma_{s,N}^{\disc,\min}$ and $\Sigma_{s,N}^{\disc,\max}$
for the linear difference equation
$u_{n+1}=A_nu_n$, $n=0,\dots,N-1$,
using the algorithm from \cite[Section 4.2]{behu24b}
with buffer intervals of length $B$
\State Normalize $\Sigma_s^{\RK, \min} \coloneqq h^{-1} \Sigma_{s,N}^{\disc, \min}$ and
$\Sigma_s^{\RK, \max} \coloneqq h^{-1} \Sigma_{s,N}^{\disc, \max}$ 
\State Color the box $\cB_{\bi}$  according to
$\Sigma_s^{\RK,\min}$ and $\Sigma_s^{\RK,\max}$, respectively
\EndFor
\end{algorithmic}}
\end{algorithm}
 
In Algorithm \ref{alg6} we use 
$\mathrm{RK}(f,h,x_n)$ to denote one step of the fourth-order explicit 
  Runge--Kutta scheme with step size $h$, applied to
  \eqref{eq3:contnonlin} with initial value $x_n$.
For nonautonomous ODEs, the starting time $nh$ is understood
implicitly in the notation $\mathrm{RK}(f,h,x_{n})$. 
  Correspondingly, we consider the discrete system 
\[
x_{n}= \mathrm{RK}(f,h,x_{n-1}), \quad n=1,\ldots,N
\]
and the solution operator 
$\Phi_h(n,m;x_0)$, $n \ge m \ge 0$,  of the linearized system
    \begin{equation*} \label{eq3:RKlineq}     
      u_{n+1}= \frac{\partial \mathrm{RK}}{\partial x}(f,h,x_n) u_n,\quad
      n = 0,\dots,N-1.      
    \end{equation*}
Then we replace the flow $\Phi$ in \eqref{eq3:limtheta}
  by the solution operator $\Phi_h$, i.e., we compute
    \begin{equation*} \label{eq3:RKlinang}
      \alpha_s^{\mathrm{RK}}(h,T,V;x_0) = \frac{1}{T} \sum_{j=1}^N \ang(\Phi_h(j-1,0;x_0)V,\Phi_h(j,0;x_0)V).
    \end{equation*}
  
As in Definition \ref{def2:angspec} we define the $s$-dimensional 
\textbf{approximate angular spectrum} by
\begin{equation*} \label{eq2:defRK}
\Sigma_s^{\RK}(h,T;x_0) = \mathrm{cl}\left\{\alpha_s^{\RK}(h,T,V;x_0): V\in\cG(s,d)\right\}.
\end{equation*} 
Its extremal values are given by
\[
  \Sigma_s^{\RK, \min}(h,T;x_0) =\inf \ \Sigma_s^{\RK}(h,T;x_0), \quad
  \Sigma_s^{\RK, \max}(h,T;x_0) =\sup \ \Sigma_s^{\RK}(h,T;x_0),
\]
and
\[
\Sigma_s^{\RK, \diff}(h,T;x_0) = \Sigma_s^{\RK, \max}(h,T;x_0) -
\Sigma_s^{\RK, \min}(h,T;x_0).
\]
In practical computations the whole set $\cG(s,d)$ is replaced by approximate trace
spaces which are computed from an approximate dichotomy spectrum obtained
at finite time.  In the following we do not discuss the 
errors induced by these additional approximations.

\subsection{Application to an autonomous planar flow}
\label{sec4.3}
Our first example in continuous time is inspired by a nonautonomous
fluid flow model introduced in \cite{FLS10}.
We consider the two-dimensional autonomous ODE model, defined on
$\cB= [0,2\pi]\times [0,\pi]$ 
\begin{equation}\label{strom}
\frac \partial {\partial t}\begin{pmatrix}u_1\\u_2\end{pmatrix}
= \begin{pmatrix}
  \cos(u_2^2)-\sin(u_1)\cos(u_2) + 
  \frac {\sin(u_1)}{\left(\left(\sin(u_1)\sin(u_2)-\frac{u_2}2+\frac
  \pi 4\right)^2 + 1\right)^4}\\\cos(u_1)\sin(u_2)
\end{pmatrix}.
\end{equation}
The left and right boundaries of $\cB$ are identified, so that the
system is considered on $(\R/2\pi\Z)\times[0,\pi]$. 

Note that $\sin(0) = \sin(\pi) = 0$; thus solutions cannot leave $\cB$
through the upper and lower boundaries. 

We choose $N=2000$ and $h=\frac{1}{20}$ and apply five successive
steps of the classical fourth-order Runge--Kutta scheme with step size
$\frac{1}{100}$ to obtain a numerical approximation of the $h$-flow.  
In Algorithm \ref{alg6}, $\mathrm{RK}$ denotes this composition, and $\frac{\partial
\mathrm{RK}}{\partial x}$ its derivative. 
We then apply this algorithm on an $L\times L$ grid,
with $L=10^3$. Numerical tests indicate that for almost all initial points, the
only spectral fiber is the entire space $\R^2$. Thus, we omit the
computationally expensive calculation of the dichotomy spectrum and its
corresponding trace spaces 
and solve only optimization problems in line $6$ of Algorithm \ref{alg6}.
Figure \ref{Fstrom1} shows the resulting angular maps in the upper
row and corresponding histograms, using the same color scheme, in the
lower row. Figure \ref{Fstrom2} displays the pointwise difference between the maximal and
minimal angular spectral values shown in Figure \ref{Fstrom1}.
A logarithmic color scale is used to make small differences more visible. 
 
\begin{figure}[hbt]
\begin{center}
      \includegraphics[width=0.99\textwidth]{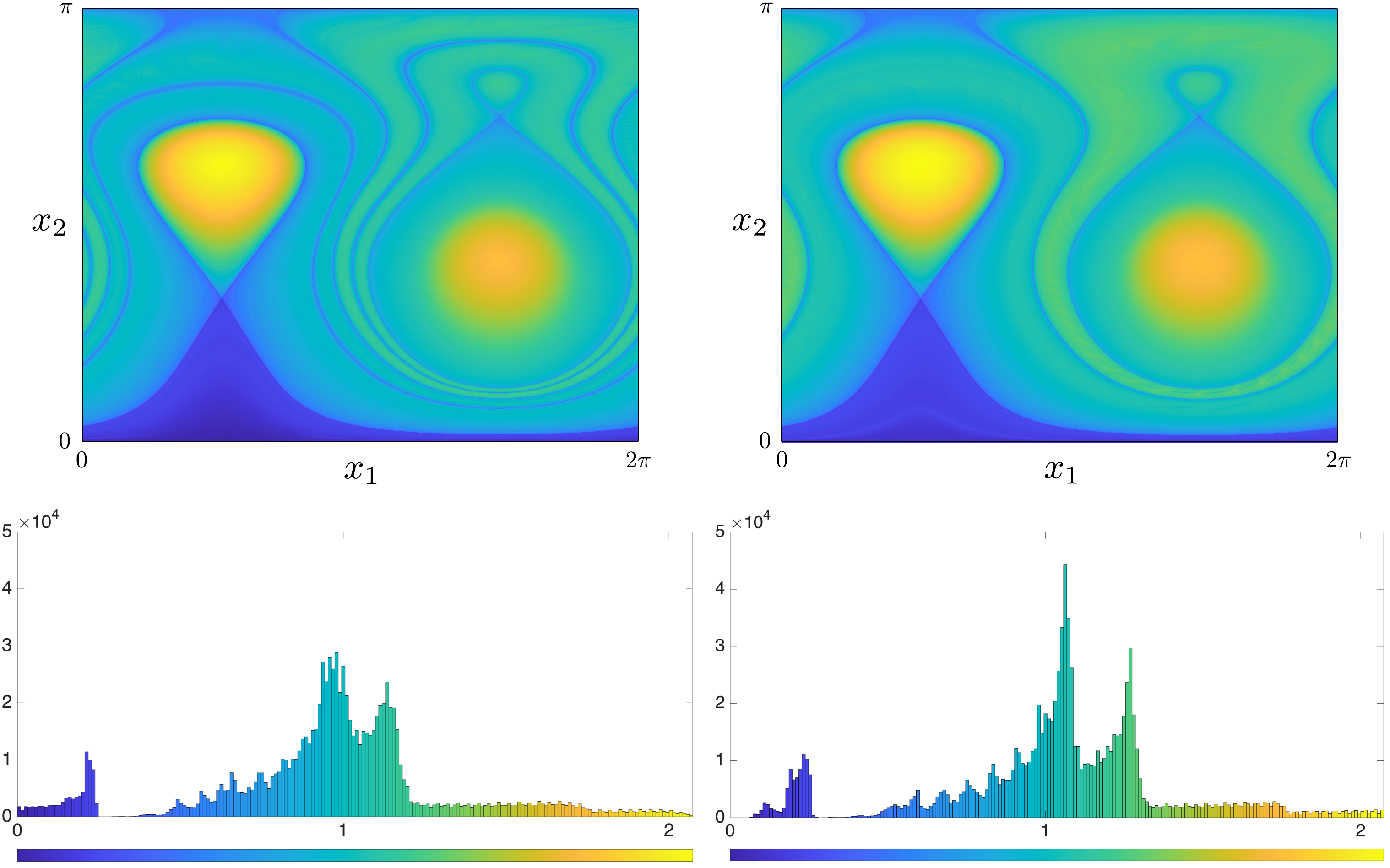}
\end{center}
\caption{Angular maps of the minimum (left) and maximum (right)
  angular spectral values for \eqref{strom}. 
  The corresponding histograms are shown in the
  lower row.\label{Fstrom1}} 
\end{figure}

\begin{figure}[hbt]
\begin{center}
\includegraphics[width=0.49\textwidth]{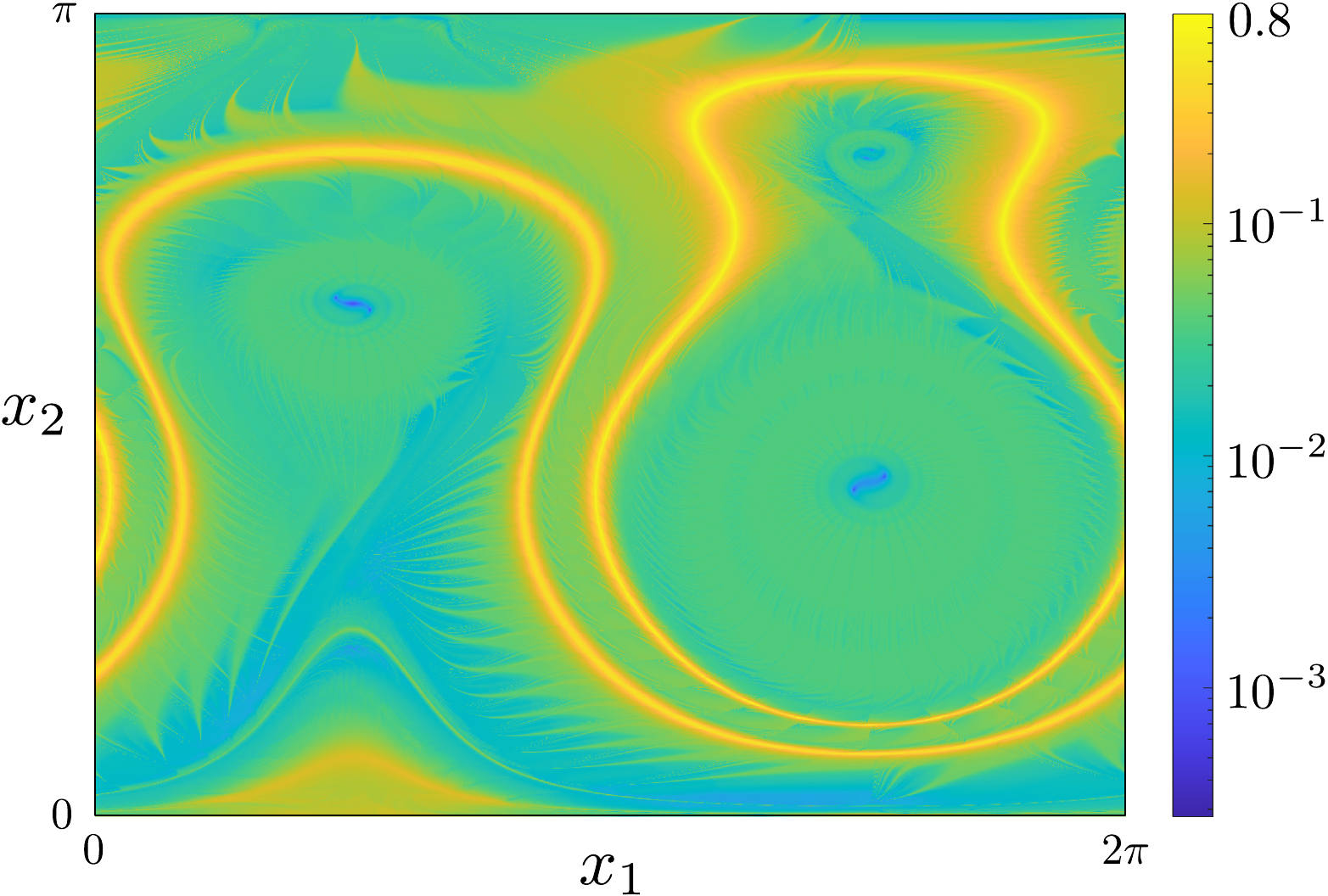}
\end{center}
\caption{Pointwise difference between the maximal and minimal angular
  spectral values shown in Figure \ref{Fstrom1}. The colors are
  displayed on a logarithmic scale. \label{Fstrom2}}  
\end{figure}

\subsection{Application to the $3D$-Lorenz system}
\label{sec_lorenz}
Our second example is the well-known three-dimensional Lorenz system 
\begin{equation}\label{lorenz}
\frac \partial {\partial t}
\begin{pmatrix}
u_1\\u_2\\u_3
\end{pmatrix}
=
\begin{pmatrix}
\sigma(u_2-u_1)\\\rho u_1 - u_2 - u_1 u_3\\ u_1 u_2 - \beta u_3
\end{pmatrix}
\quad \text{with parameters}\quad \sigma = 10,\ \rho = 28,\ \beta = \tfrac 83.
\end{equation}
We first take the $h$-flow of the Lorenz system with $h = \frac 1{20}$
and compute an orbit that is close to the Lorenz attractor. 
For the initial conditions tested, the computed finite-time spectral
intervals of the dichotomy spectrum show little dependence on the
initial point. Our numerical procedure yields three separated spectral
intervals near $0.48$, $1$, and $1.05$. 

Next, we take a grid of size $L^3$ with $L=100$ in $\cB =
[-20,30]\times[-20,30]\times[-10,50]$. For all initial conditions
chosen at the midpoints of the grid cells, the computed orbits
eventually approach the Lorenz attractor. Thus, we observe only
one-dimensional spectral fibers in the Lorenz system. 
This allows us to avoid the costly computation of the dichotomy
spectrum and its associated trace spaces at every grid point.
Instead, as in Section \ref{sec3.1a}, we approximate the relevant
trace spaces by combining forward and backward iterations.  
Figure \ref{FL1} shows the three-dimensional angular maps for $N=200$
and $N=2000$, together with the corresponding histograms.

\begin{figure}[hbt]
\begin{center}
\includegraphics[width=0.99\textwidth]{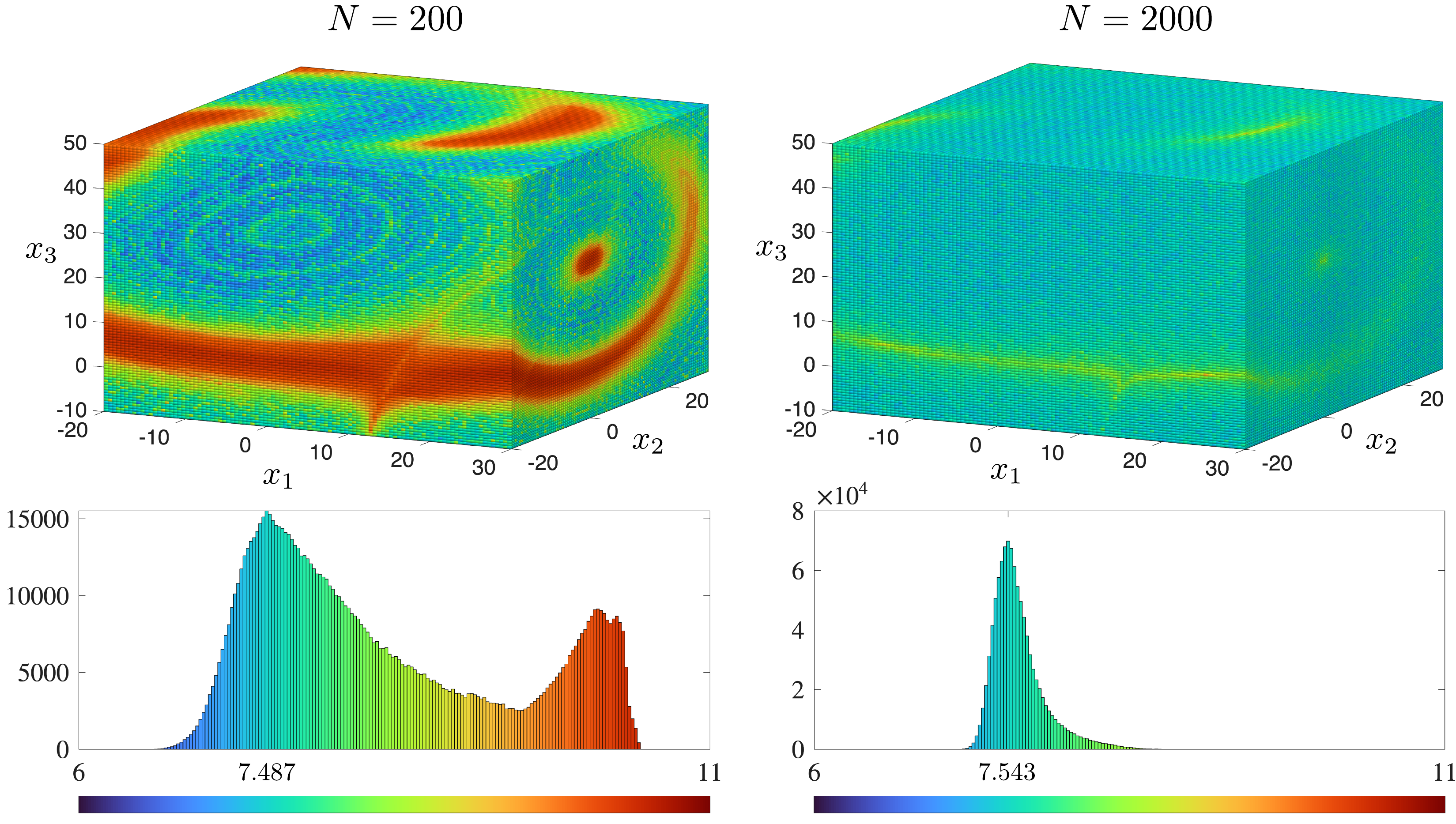}      
\end{center}
\caption{Spectral angular values for $s=1$ computed by forward
  iteration of the $h$-flow of
  the Lorenz system \eqref{lorenz}, with $h=\frac 1{20}$. 
  The results for $N=200$ and
  $N=2000$ are shown on the  left and right, respectively, together
  with the corresponding histograms.   \label{FL1}} 
\end{figure}

For $N=2000$, we compute all angular spectral values for $s=1$ and
$s=2$. The corresponding histograms are shown in Figure 4.4. We
present only the histograms and label each peak with the trace space
in which the corresponding angular spectral value is attained. 

\begin{figure}[hbt]
\begin{center}
\begin{tabular}{cc}
$s = 1$ & $s=2$\\[-1mm]
\includegraphics[width=0.47\textwidth]{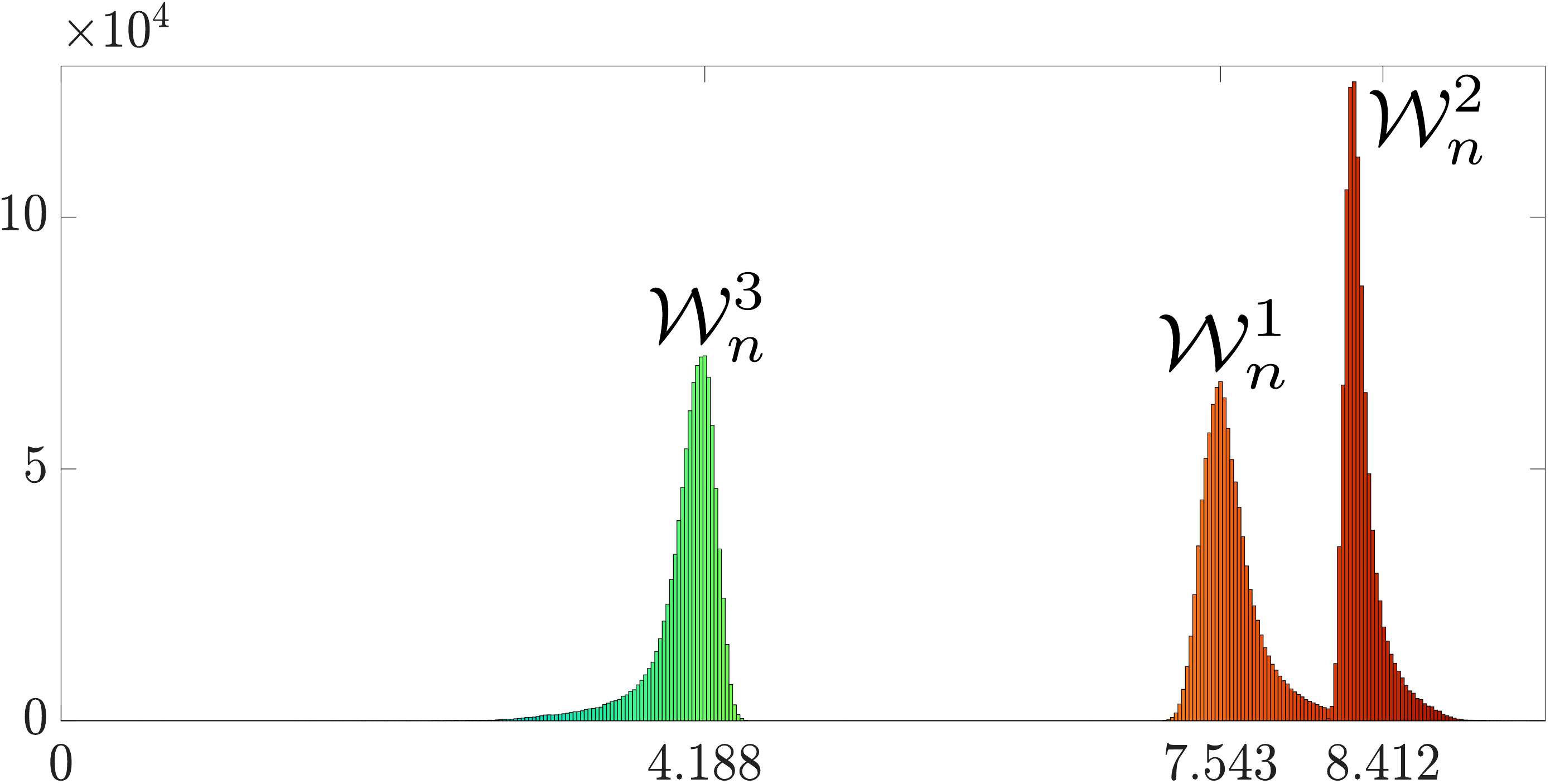}&
\includegraphics[width=0.46\textwidth]{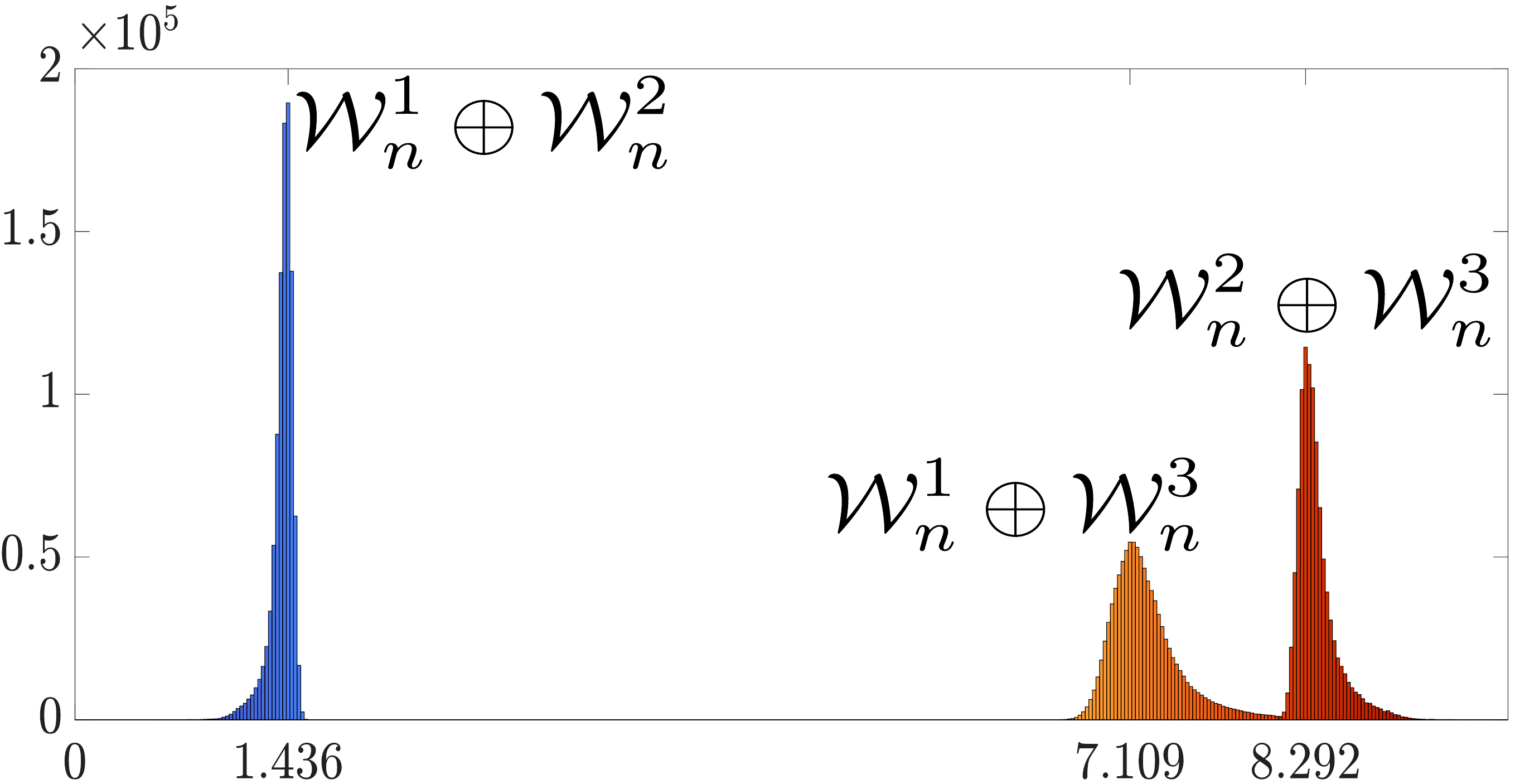}
\end{tabular}      
\end{center}
\caption{Histograms of the angular map for the three-dimensional
  Lorenz system, computed with  $N=2000$, for $s=1$ (left) and
  $s=2$ (right). Each peak is  labeled with its associated trace space.  \label{FL2}}   
\end{figure}
 
\subsection{Convergence of angular spectra for time-discretized systems} 
\label{sec4.4}
In the following we analyze the approximation error
 \begin{equation*} \label{eq3:errorquad}
   e(h,T,V)= | \alpha_s^{\step}(h,T,V)- \alpha_s^{\cont}(T,V)|
 \end{equation*}
 of the terms in \eqref{eq3:limtheta} and \eqref{eq3:defcont}. To simplify notation, we omit the
 dependence on the initial data $x_0$, i.e., we write
 \begin{align*}
   \alpha_s^{\step}(h,T,V)= \alpha_s^{\step}(h,T,V;x_0), \quad  \alpha_s^{\cont}(T,V)=\alpha_s^{\cont}(T,V;x_0).
   \end{align*}
  In Remark \ref{rem3:x0} we will indicate when all the estimates hold uniformly in $x_0\in \cB$.
 Our goal is to prove the following theorem.
 \begin{theorem} \label{th3:unifapproxT}
   Let $A\in C^1_b([0,\infty),\R^{d,d})$, i.e., $A(\cdot)$ and $\dot{A}(\cdot)$ are continuous and bounded.
     Then there exist constants $C,h_0>0$ such that the estimate
          \begin{equation*} \label{eq3:esterrh}
e(h,T,V) \le C h
          \end{equation*}
          holds for all $V \in \cG(s,d)$, for all $T>0$ and for all $h = \frac{T}{N} \le h_0$, where $N \in \N$.
 \end{theorem}
For the proof we use a linear algebra lemma inspired by \cite[Theorem 2.5.1]{GvL2013}.
 \begin{lemma} \label{lem3:spectral}
   Let $A\in \R^{d,d}$, let $P$ be an orthogonal projector in $\R^d$ and let $\|\cdot\|$ be the spectral
   norm. Then the following equality holds
   \begin{equation} \label{eq3:specequal}
     \|(I-P)A P + PA^{\top}(I-P)\| = \| (I-P)AP\|.
   \end{equation}
 \end{lemma}
 \begin{proof}
   Recall $\|B\|=\|B^{\top}\|$ for $B \in \R^{d,d}$. Further, by the orthogonality of $P$,
   \begin{equation*}
     \|(I-P)x + Py \|^2 = \|(I-P)x\|^2 + \| Py\|^2, \quad \forall \ x,y \in \R^d.
   \end{equation*}
   Using this and $P^{\top}=P$,  we obtain for all $x \in \R^d$
   \begin{equation*}
     \begin{aligned}
       \|((I-P)AP &+PA^{\top}(I-P))x\|^2 = \|(I-P)APx\|^2 + \|PA^{\top}(I-P)x\|^2 \\
       & \le \|(I-P)AP\|^2 \|Px\|^2 + \|PA^{\top}(I-P)\|^2 \|(I-P)x\|^2 \\
       & = \|(I-P)AP\|^2(\|Px\|^2+ \|(I-P)x\|^2) = \|(I-P)AP\|^2 \|x\|^2;
     \end{aligned}
   \end{equation*}
   hence, $\|(I-P)AP + PA^{\top}(I-P)\| \le \|(I-P)A P\|$. Next, there exists some $x \in \R^d$ with $\|x\|=1$
   and $\|(I-P)APx\| = \|(I-P)AP\|$. Then we have $\|Px\| \le1$ and
   \begin{equation*}
     \begin{aligned}
       \|(I-P)AP\| & = \|(I-P)AP(Px)\|= \|((I-P)AP + P A^{\top}(I-P))(Px) \|\\
       & \le \|(I-P)AP+ PA^{\top}(I-P)\|\|Px\| \le \|(I-P)AP + P A^{\top}(I-P)\|.
     \end{aligned}
   \end{equation*}
  Combining both estimates leads to the equality in \eqref{eq3:specequal}.
   \end{proof}
 The following proposition sharpens the result \cite[Theorem 3.4]{behu24a} on derivatives of
 principal angles and simplifies the proof.
 \begin{proposition} \label{prop3:estdiffq}
   Let $A \in C^1_b([0,\infty), \R^{d,d})$. Then there exist constants $C,h_0>0$ such that for all
     $\tau \ge 0$, $V\in \cG(s,d)$ and $0< h \le h_0$,
     \begin{equation} \label{eq3:estdiffq}
       \left| \frac{1}{h} \ang(\Phi(\tau+h,\tau)V,V)- \|(I-P_V)A(\tau)P_V\| \right| 
       \le C h,
     \end{equation}
where $P_V$ denotes the orthogonal projector onto $V$, $\|\cdot\|$ is the spectral norm, and $\Phi$ is the
solution operator of \eqref{eq3:contvari}.
 \end{proposition}
 \begin{proof}
   Choose an orthonormal basis of $V$, i.e., some $W_0 \in \R^{d,s}$ with $\range(W_0)=V$ and $W_0^{\top}W_0= I_s$.
   For a fixed $\tau \ge 0$, define $W(t)= \Phi(\tau+t,\tau)W_0$ so that $\range(W(t))= \Phi(\tau+t,\tau)V$ holds.
   The matrix $W(t)$ and its derivative solve the initial value problems
   \begin{align*}
     \dot{W}(t) &=A(t+\tau) W(t), \quad  W(0)= W_0, \\
     \ddot{W}(t) & = \left[ \dot{A}(t+\tau)+ A(t+\tau)^2 \right] W(t), \quad \dot{W}(0)=A(\tau) W_0 =: \dot{W}_0.
   \end{align*}
   Let $\|A(\tau)\| + \|\dot{A}(\tau)\| \le C$ for all $\tau \ge 0$. Then a Gronwall estimate yields 
   $\|W(t)\| \le e^{tC}\|W_0\| \le e^C$ for $0\le t \le 1$ as well as 
   \begin{equation*}
     \|\dot{W}(t)\| \le C e^C, \quad \|\ddot{W}(t)\| \le C(1+C) e^C =: C_{\star}.
   \end{equation*}
   By a Taylor expansion we obtain for $0 \le h \le 1$
   \begin{align*}
     \|W(h)-W_0 - h \dot{W}_0\|= \| h^2\int_0^1(1-\eta)\ddot{W}(\eta h) d\eta \| \le \frac{C_{\star}}{2} h^2.
   \end{align*}
   We write this in the usual way as $W(h) = W_0 + h \dot{W}_0 + \mathcal{O}(h^2)$, where $\mathcal{O}(h^2)$ holds
   with a constant that is uniform in $\tau \ge0$ and $V \in \cG(s,d)$. Correspondingly, we have
   $W(h)^{\top}= W_0^{\top} +  h \dot{W}^{\top}_0 + \mathcal{O}(h^2)$. This implies
     \begin{equation*} W(h)^{\top}W(h)= I_s + h(W_0^{\top}\dot{W}_0 + \dot{W}_0^{\top}W_0) +\mathcal{O}(h^2).
     \end{equation*}
     Hence, there exists a $0<h_0 \le 1$, independent of $\tau \ge 0$ and $V \in \cG(s,d)$, such that
     the inverse exists for $0\le h \le h_0$ and satisfies
       \begin{equation*} \left(W(h)^{\top}W(h)\right)^{-1}= I_s - h(W_0^{\top}\dot{W}_0 + \dot{W}_0^{\top}W_0) +\mathcal{O}(h^2).
     \end{equation*}
   The orthogonal projectors associated to $V$ and to $\Phi(\tau+h,\tau)V$, $0 \le h \le h_0$, are
   \begin{align*}
     P_V  &  = W_0 W_0^{\top}, \\
     P_{\Phi(\tau+h,\tau)V}& = W(h) \left(W(h)^{\top}W(h)\right)^{-1}W(h)^{\top}.
   \end{align*}
   Using the previous expansions, we obtain
   \begin{equation} \label{eq5:expandnextV}
   \begin{aligned}
     P_{\Phi(\tau+h,\tau)V}& = W_0 W_0^{\top}+ h \left[ \dot{W}_0W_0^{\top} - W_0(W_0^{\top} \dot{W}_0 + \dot{W}_0^{\top} W_0)
       W_0^{\top}+ W_0 \dot{W}_0^{\top} \right] + \mathcal{O}(h^2)\\
     & = P_V+ h \left[ (I-W_0W_0^{\top})\dot{W}_0W_0^{\top} + W_0 \dot{W}_0^{\top}(I-W_0W_0^{\top}) \right] + \mathcal{O}(h^2)\\
     & = P_V + h \left[ (I-P_V)A(\tau)P_V + P_V A(\tau)^{\top}(I-P_V) \right] + \mathcal{O}(h^2).
   \end{aligned}
   \end{equation}
   For later reference, let us note that this implies
   \begin{equation} \label{eq3:projhest}
     \|P_{\Phi(\tau+h,\tau)V} - P_V\| = \mathcal{O}(h).
   \end{equation}
   Let us  abbreviate $\varphi(h) = \ang(\Phi(\tau+h,\tau)V,V)$ in the following. With Lemma \ref{lem3:spectral}
   and formula \eqref{eq2:charsin} we obtain from the expansion above
   \begin{align*}
    & \left| \frac{1}{h}  \sin(\varphi(h))-\|(I-P_V)A(\tau) P_V \|  \right| \\
     & = \left| \frac{1}{h} \|P_{\Phi(\tau+h,\tau)V}-P_V\| - \|(I-P_V)A(\tau)P_V + P_VA(\tau)^{\top}(I-P_V) \| \right| \\
     & \le \left\| \frac{1}{h}(P_{\Phi(\tau+h,\tau)V}-P_V) - \left[(I-P_V)A(\tau)P_V + P_VA(\tau)^{\top}(I-P_V)\right] \right\|\\
     & = \mathcal{O}(h).
   \end{align*}
   Finally, using $\arcsin(y)= y + \mathcal{O}(y^2)$  and $\sin(\varphi(h))= \mathcal{O}(h)$, we end up
   with
   \begin{align*}
     \frac{\varphi(h)}{h}&= \frac{1}{h} \arcsin(\sin(\varphi(h)))= \frac{1}{h}\left( \sin(\varphi(h))+ \mathcal{O}(\sin(\varphi(h))^2)\right)
     \\
     & = \frac{1}{h}\left( \sin(\varphi(h))+ \mathcal{O}(h^2)\right) = \|(I-P_V)A(\tau)P_V\| + \mathcal{O}(h).
     \end{align*}
   \end{proof}
 \begin{proof}[Proof of Theorem \ref{th3:unifapproxT}]
   Let us abbreviate $V_{\tau}= \Phi(\tau,0)V$  and
   $E(\tau) = (I-P_{V_{\tau}})A(\tau) P_{V_{\tau}}$.
   With $T=N h$ we obtain by substituting  $\tau=(j-1+r)h$
      \begin{align*}
    & e(h,T,V) = \frac{1}{hN} \Big| \sum_{j=1}^N \ang(V_{(j-1)h},V_{jh}) - 
     \sum_{j=1}^N \int_{(j-1)h}^{jh} \|E(\tau) \| d\tau \Big| \\
     & =  \frac{1}{N} \Big| \frac{1}{h} \sum_{j=1}^N \ang(V_{(j-1)h},V_{jh}) - 
     \sum_{j=1}^N \int_{0}^{1} \|E((j-1+r)h)\| dr \Big|\\
     & \le \frac{1}{N}\left[ \sum_{j=1}^N \Big|\frac{1}{h} \ang(V_{(j-1)h},\Phi(jh,(j-1)h)V_{(j-1)h})-
       \|E((j-1)h) \| \Big| \right. \\
      & \left. + \sum_{j=1}^N \int_0^1  \|E((j-1)h)- E((j-1+r)h)\| dr \right].
        \end{align*}
   The terms in the first sum are bounded by $\mathcal{O}(h)$ due to \eqref{eq3:estdiffq} with
   $\tau=(j-1)h$, $V=V_{(j-1)h}$. Recall that all $\mathcal{O}$-estimates are uniform in $\tau$
   and $V$.
   The terms in the second sum are also $\mathcal{O}(h)$ due to \eqref{eq3:projhest} and the boundedness and
   Lipschitz boundedness of $A(\cdot)$:
   \begin{align*}
     \|E(\tau)-E(\tau+rh)\| & \le \|(I-P_{V_{\tau}})- (I-P_{V_{\tau+rh}})\| \|A(\tau) P_{V_{\tau}}\| \\
     & +\| I-P_{V_{\tau+rh}}\| \|A(\tau)-A(\tau+rh)\|\| P_{V_{\tau}}\| \\
     &+
     \| I-P_{V_{\tau+rh}}\| \|A(\tau+rh)\|\| P_{V_{\tau}}-P_{V_{\tau+rh}}\|\\
     & \le C h, \quad \forall\, \tau \ge 0, \, 0\le h \le h_0, \, 0 \le r \le 1.
   \end{align*}
   In view of the prefactor $\frac{1}{N}$ we finally obtain $e(h,T,V) = \mathcal{O}(h)$.
   \end{proof}
 
 As in \eqref{eq2:angrange} let us introduce the angular range generated by a subspace $V \in \cG(s,d)$ for a  fixed stepsize
 $h>0$:
 \begin{equation} \label{eq4:limindiv}
         I_s^{\step}(h,V) =
     \left[\varliminf_{N\to \infty} \alpha_s^{\step}(h,Nh,V) , \varlimsup_{N\to \infty}\alpha_s^{\step}(h,Nh,V) \right].
  \end{equation}
  
  The result of Theorem \ref{th3:unifapproxT} enables us to
  prove a strong approximation of the (outer) angular spectrum $\Sigma^{\cont}_s$  (see \eqref{eq3:defangspec})
  by its discrete counterpart
 \begin{equation*} \label{eq3:defangspeccont}
     \Sigma_s^{\step}(h) = \mathrm{cl}\Big[ \bigcup_{V \in \cG(s,d)} I_s^{\step}(h,V)\Big]
 \end{equation*}
 with respect to the Hausdorff metric.
 \begin{theorem} \label{th4:approxspechaus}
   Let $A \in C^1_b([0,\infty),\R^{d,d})$. Then there exist constants $C,h_0>0$ such that
     \begin{equation} \label{eq4:specerror}
              d_H(\Sigma_s^{\step}(h),\Sigma_s^{\cont})\le C h, \quad \forall \, 0<h \le h_0,
     \end{equation}
     where $d_H$ denotes the Hausdorff distance.
 \end{theorem}

 \begin{proof} We will show the estimate
   \begin{equation*} \label{eq4:estindiv}
     d_H(I_s^{\step}(h,V),I_s^{\cont}(V)) \le C h, \quad \forall \, 0<h \le h_0, V \in \cG(s,d)
   \end{equation*}
   for the angular ranges from \eqref{eq3:contrange} and \eqref{eq4:limindiv}.
   The estimate \eqref{eq4:specerror} then follows by taking the union over $V \in \cG(s,d)$ and by a subsequent
   closure argument.
   As another preparation we note that $E(t) = (I-P_{\Phi(t,0)V})A(t) P_{\Phi(t,0)V}$ satisfies for $T,\tau >0$ 
   \begin{equation} \label{eq4:estalphaT}
     \begin{aligned}
     \left|\alpha_s^{\cont}(T, V) -\alpha_s^{\cont}(\tau,V)\right| & \le \left| \frac{1}{T} - \frac{1}{\tau} \right|\int_0^T\|E(\zeta)\| d \zeta + \Big|
     \frac{1}{\tau} \int_{T}^{\tau} \|E(\zeta)\| d \zeta \Big|  \\
     & \le \frac{2}{\tau}|T- \tau| \sup_{\zeta\ge 0}\|E(\zeta)\| \le \frac{2 C}{\tau} |T - \tau|.
     \end{aligned}
   \end{equation}
   Consider first $\theta \in I_s^{\cont}(V)$. By Lemma \ref{lem4:accucont} there exists a sequence $T_n \to \infty$ with
   $\theta= \lim_{n \to \infty} \alpha_s^{\cont}(T_n,V) $. For $0 < h \le h_0$ choose $n_0=n_0(h)$ such that
   $|\alpha_s^{\cont}(T_n,V) - \theta| \le h$ for all $n \ge n_0$. Then we set $N(h,n) =
   \lfloor \frac{T_n}{h} \rfloor$ so that $|T_n - h N(h,n)| \le h$ and $N(h,n) \to \infty$ as $n \to \infty$ hold. An
   application of  Theorem \ref{th3:unifapproxT} and of \eqref{eq4:estalphaT} yields for $n \ge n_0$
   \begin{align*}
     | \alpha_s^{\step}(h,h N(h,n),V)- \theta| & \le |\alpha_s^{\step}(h,h N(h,n),V)- \alpha_s^{\cont}(hN(h,n),V) | \\
     & + |\alpha_s^{\cont}(h N(h,n),V)- \alpha_s^{\cont}(T_n,V)| +|\alpha_s^{\cont}(T_n,V)- \theta| \\
     & \le Ch + \frac{2C}{T_n}|T_n-h N(h,n)| + h \le Ch.
   \end{align*}
   Therefore, we obtain the inequalities
   \begin{align*}
     \varliminf_{k \to \infty}\alpha_s^{\step}(h,kh,V) - Ch & \le  \varliminf_{n \to \infty}\alpha_s^{\step}(h,hN(h,n),V) - Ch \le \theta \\
     & \le \varlimsup_{n \to \infty}\alpha_s^{\step}(h,hN(h,n),V)+Ch  \le \varlimsup_{k \to \infty} \alpha_s^{\step}(h,kh,V) + Ch.
   \end{align*}
   Thus, we have shown $\mathrm{dist}(I_s^{\cont},I_s^{\step}(h)) \le Ch $ for the Hausdorff semidistance.

   The proof of the converse estimate is somewhat analogous. For $\theta \in I_s^{\step}(h)$, Lemma \ref{lem2:accu}
  yields a sequence $N(h,n) \to \infty$ with $\lim_{n \to \infty}\alpha_s^{\step}(h,h N(h,n),V)=\theta$.
   Then we choose $n_0=n_0(h)$ with $|\alpha_s^{\step}(h,hN(h,n),V)- \theta | \le h$ for all $n \ge n_0$.
  Thus  Theorem \ref{th3:unifapproxT} leads for $n \ge n_0$ to the estimate
   \begin{align*}
     |\theta - \alpha_s^{\cont}(hN(h,n),V)| & \le |\theta - \alpha_s^{\step}(h,hN(h,n),V)| \\
     & + |\alpha_s^{\step}(h,hN(h,n),V) - \alpha_s^{\cont}(hN(h,n),V)| \le h + C h.
   \end{align*}
   As above, we obtain
   \begin{align*} \varliminf_{T \to \infty} \alpha_s^{\cont}(T,V) - Ch &\le \varliminf_{n \to \infty}\alpha_s^{\cont}(hN(h,n),V) - Ch
     \le \theta \\
     & \le \varlimsup_{n \to \infty} \alpha_s^{\cont}(hN(h,n),V) + Ch
       \le \varlimsup_{T \to \infty} \alpha_s^{\cont}(T,V) + Ch. 
   \end{align*}
   Thus, we conclude $\mathrm{dist}(I_s^{\step}(h),I_s^{\cont}) \le Ch$, which finishes our proof.
   \end{proof}
 \begin{remark}
   \label{rem3:x0}Recall the dependence of the linear system \eqref{eq3:contvari} on the data $x _0 \in \cB$.
   If $A(\cdot,x_0)$ and its first derivative $\dot{A}(\cdot,x_0)$ have bounds that are uniform for
   $t \ge 0$ and  $x_0\in \cB$,
   then the estimates in Theorems \ref{th3:unifapproxT} and \ref{th4:approxspechaus} hold
   uniformly for $x_0 \in \cB$.
   \end{remark}


\subsection{A simplified algorithm for the continuous-time case}
 
Recall the continuous-time system
\begin{equation*}\label{ODE1}
\dot x(t) = f(t,x(t)),\quad t \in \R 
\end{equation*}
and its solution operator $\Psi(t_1,t_0,x_0)$ from \eqref{eq3:contnonlin}.
For sufficiently small $h>0$, we use the $h$-step map
$\Psi(nh, (n-1)h; \cdot)$ in place of $F_{n-1}$ of the previous algorithms.
The following simple algorithm uses only points on the trajectory.
\begin{algorithm}
\caption{A simplified algorithm for continuous-time
  systems in case $s=1$\label{alg3}}
  \textcolor{header1}{
  \begin{algorithmic}[1] 
\For {$\bi\in\cL^d$}
\State Choose the midpoint $x_0 \in \cB_{\bi}$
  \For {$n = 1,\dots,N+1$}
  \State $x_{n} = \Psi(nh,(n-1)h, x_{n-1})$
  \EndFor
  \State $\Gamma_{\bi} =\displaystyle \frac 1{h N} \sum_{\ell = 1}^N \ang(x_{\ell-1}-
  x_{\ell}, x_{\ell} - x_{\ell+1})$
  \State Color the box $\cB_{\bi}$ with the color
  that corresponds to the angle $\Gamma_{\bi}$  
  \EndFor
  \end{algorithmic}} 
\end{algorithm}                              

In Section \ref{sec5.6} we will prove that, for autonomous systems under
suitable assumptions, the resulting angular range is close to the
angular range associated with the flow direction. 

When applying Algorithm \ref{alg3} to the ODE \eqref{strom}, we compute
the differences using a continuous lift of the first coordinate to
$\R$. 
Figure \ref{Fstrom3} shows the resulting angular values along the flow
direction and the corresponding histogram in the left panels.
The angular map looks similar to the upper-right diagram in Figure
\ref{Fstrom1}, where we compute the maximum angular spectral values by
solving an optimization problem. The right panel of Figure
\ref{Fstrom3} gives the pointwise absolute difference between the
maximum angular spectral values and the angular values along the flow
direction. 
 
\begin{figure}[hbt]
\begin{center}
      \includegraphics[width=0.99\textwidth]{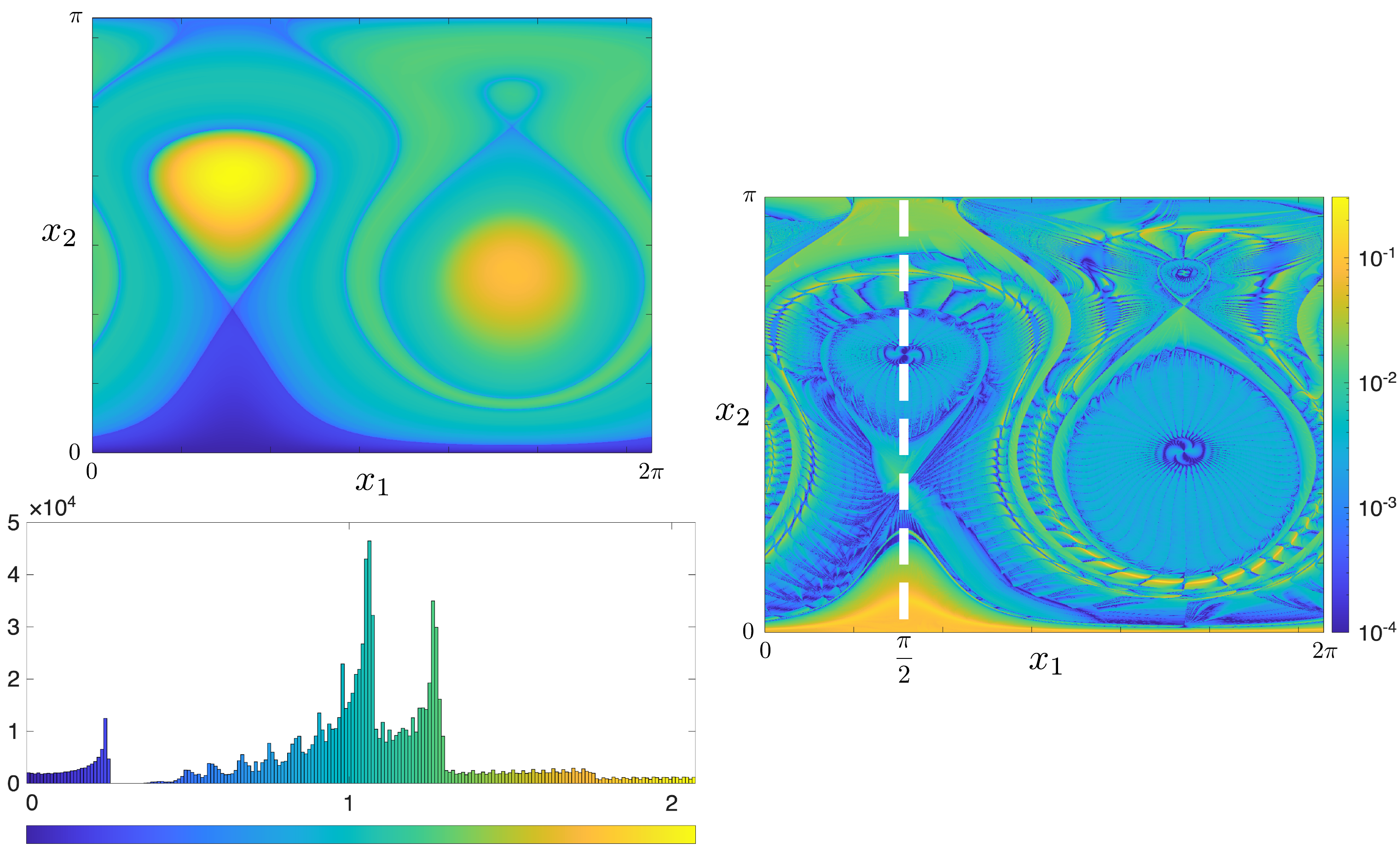}
\end{center}
\caption{Left panels: Angular map of the flow direction for \eqref{strom}
and the corresponding histogram. The right panel shows, on a
logarithmic color scale, the pointwise absolute difference between the
angular map in the upper-left panel and the maximum angular spectral
values shown in the upper-right panel of Figure
\ref{Fstrom1}. The white dashed line marks the section $x_1=\frac \pi 2$
investigated in Figure \ref{Fstrom4}.
\label{Fstrom3}}
\end{figure}

To examine these differences more closely, we consider initial points
$x_0=(\frac \pi 2,x_2)^\top$, $x_2\in[0,\pi]$, along the white dashed line in
Figure \ref{Fstrom3}. We compute the minimal and maximal angular values and
compare them with the values obtained from Algorithm \ref{alg3}, using
$T=5000$ and $h=\frac 1{20}$. Each step is evaluated by five Runge--Kutta
steps of size $\frac{1}{100}$. Starting from a uniform grid in $x_2$, we
locally maximize the difference between the maximal and minimal
angular values with respect to $x_2$.
The results in Figure \ref{Fstrom4} reveal three pronounced peaks in this
difference. Near $x_2=0.236$, the flow direction yields approximately
the minimal angular value, whereas near $x_2=2.6$ it yields
approximately the maximal value. Near $x_2=0.733$, its angular value
lies strictly between the two extrema.
Thus, the mean rotation along the flow direction does not capture
the full range of mean rotations of infinitesimal perturbations.
The angular maps identify regions where this additional variation
occurs. A detailed analysis of the associated periodic orbits will
be given in a subsequent paper.

\begin{figure}[hbt]
\begin{center}
      \includegraphics[width=0.80\textwidth]{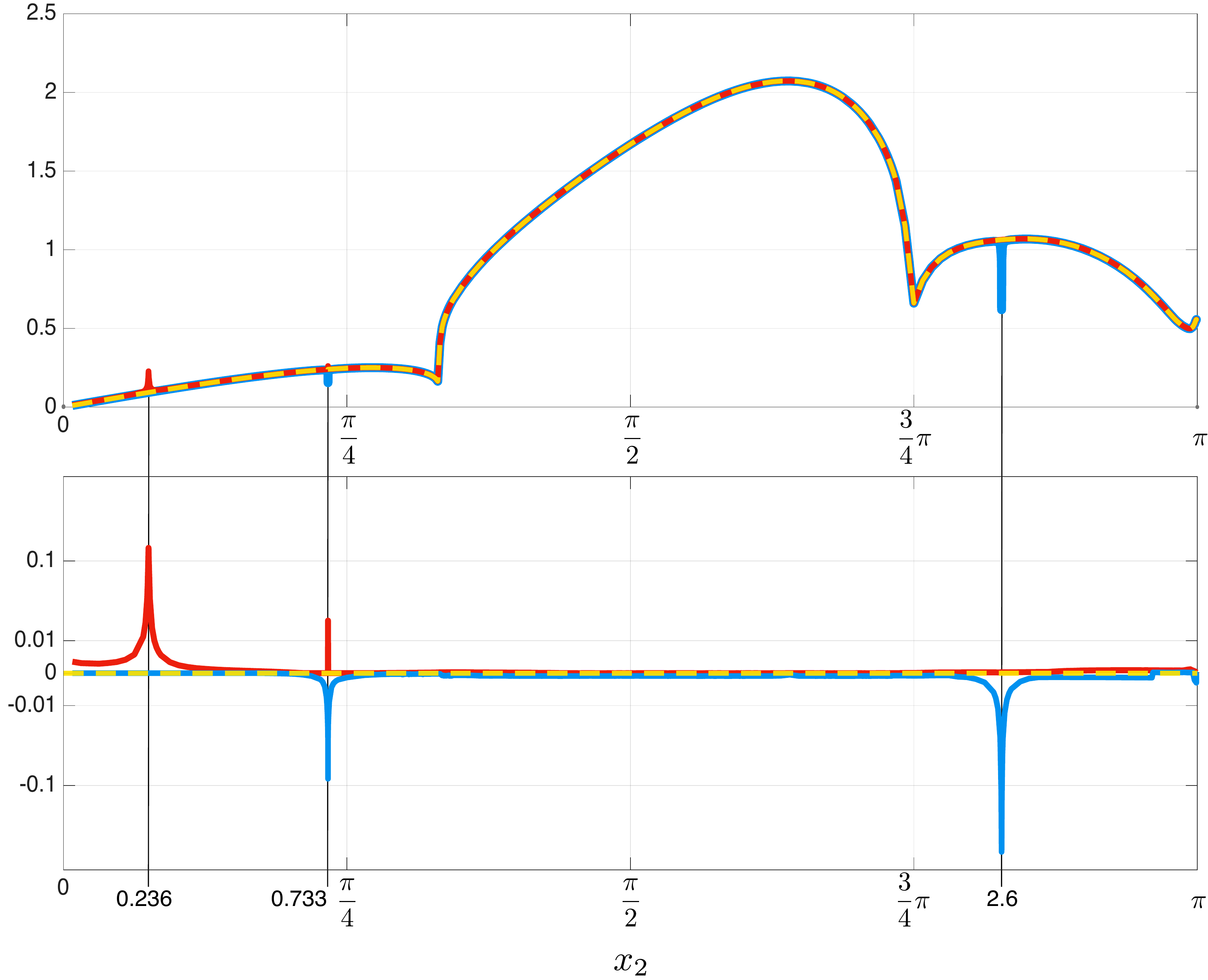}
\end{center}
\caption{Angular values for initial points $(\frac \pi 2,x_2)^\top$ 
along the white dashed
line in Figure~\ref{Fstrom3}, with $T=5000$ and $h=\frac{1}{20}$.
Top: minimal (blue) and maximal (red) angular values, compared with
the values obtained from Algorithm \ref{alg3} (yellow dashed).
Bottom: the same values after subtracting the angular value of the
flow direction, which therefore corresponds to zero.
The vertical axis uses a symmetric logarithmic scale with an
approximately linear region near zero, with transition scale $10^{-2}$. 
\label{Fstrom4}}
\end{figure}

Figure \ref{FL3} presents an application of Algorithm \ref{alg3} to
the three-dimensional Lorenz system \eqref{lorenz} with step size
$h=\frac 1{20}$ and $N=2000$ on a grid of size $L^3$, where $L=100$. 

\begin{figure}[hbt]
\begin{center}
      \includegraphics[width=0.99\textwidth]{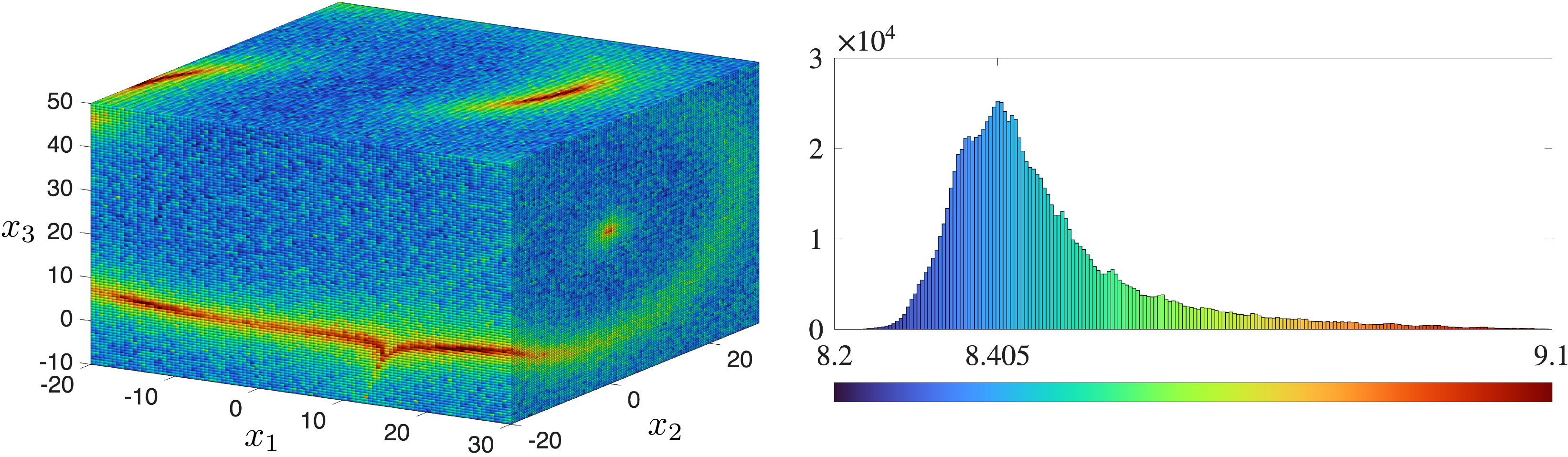}
\end{center}
\caption{Angular map of the flow direction for
  \eqref{lorenz} with $N = 2000$ and $h = \frac 1{20}$ (left) and 
  the corresponding histogram (right).  
\label{FL3}}
\end{figure}

For the orbit segments considered here, the numerical computations in
Section \ref{sec_lorenz} yield three separated finite-time spectral
intervals, with the second interval containing $1$. The corresponding
computed spectral fiber $\cW_n^2$ is associated with the flow
direction. Thus, in the left histogram of Figure \ref{FL2}, the peak
near $8.4$ corresponds to the flow direction. Accordingly, the
simplified approach of Algorithm \ref{alg3} produces a peak at
approximately the same value in Figure \ref{FL3}. 
 
 \subsection{Convergence analysis of  Algorithm \ref{alg3}}
 \label{sec5.6}
 In this section we consider Algorithm \ref{alg3} with step-size $h$ for an autonomous system
 \begin{equation} \label{eq5:ODE2}
   \dot{x}(t)=f(x(t)), \quad t \in \R, \quad x(0)=x_0.
   \end{equation} 
 Algorithm \ref{alg3} aims at approximating the time average $\alpha_1^{\mathrm{cont}}(T,V_0)$ from
 \eqref{eq3:defcont} for the special choice $V_0=\mathrm{span}(\dot{x}(0))$ by
   \begin{equation} \label{eq5:alpha1V0}
     \alpha_1^{\mathrm{diff}}(h,T,V_0)=\frac{1}{T}\sum_{j=1}^N \ang(x((j-1)h)-x(jh),x(jh)-x((j+1)h)),
     \quad h=\frac{T}{N}.
   \end{equation}
   Here and in what follows we suppress the dependence of the solution $x(t)=x(t;x_0)$ on the
   initial value. In the autonomous case, the space $V_0$ is mapped by the solution operator
   $\Phi(\tau,0)$ of the variational equation \eqref{eq3:contvari} to
   \begin{equation} \label{eq5:vareqtau}
     V_{\tau}= \mathrm{span}( \dot{x}(\tau)), \quad \dot{x}(\tau)=\Phi(\tau,0) \dot{x}(0).
   \end{equation}
   As above, we do not discuss the error caused by the approximation of the nonlinear flow
   $\Psi(\tau,0)$ in \eqref{eq5:alpha1V0} and treat only the error induced by taking differences
   along the solution path.
   The analog of Proposition \ref{prop3:estdiffq} is the following.
   \begin{proposition} \label{prop5:estdiff}
     Let $f \in C^2(\R^d,\R^d)$ and let $x(t)$, $t \ge 0$, be a nonconstant bounded solution of \eqref{eq5:ODE2}.
     Then there exist constants $C,h_0>0$ such that
     \begin{equation} \label{eq5:esttau}
       |\ang(\dot{x}(\tau -h), \dot{x}(\tau))- \ang(x(\tau-h)-x(\tau),x(\tau)-x(\tau+h))| \le C h^2
     \end{equation}
     holds for all $0<h \le h_0$ and for all  $\tau\ge h$. 
   \end{proposition}
   \begin{proof}
      In the following we write $P_v=P_{\mathrm{span}(v)}$ for the orthogonal projector onto a
     one-dimensional subspace. From  \eqref{eq2:charsin} we obtain
     \begin{equation} \label{eq5:sinest}
       \begin{aligned}
     & \big| \sin(\ang(\dot{x}(\tau -h), \dot{x}(\tau)))- \sin(\ang(x(\tau-h)-x(\tau),x(\tau)-x(\tau+h)))\big| \\
         & = \big| \|P_{\dot{x}(\tau-h)}-P_{\dot{x}(\tau)}\| - \| P_{x(\tau-h)-x(\tau)}-P_{x(\tau)-x(\tau+h)}\| \big| \\
         & \le  \|P_{\dot{x}(\tau-h)}-P_{\dot{x}(\tau)} - P_{x(\tau-h)-x(\tau)}+P_{x(\tau)-x(\tau+h)}\| .
       \end{aligned}
     \end{equation}
     Due to our assumptions we have that $\dot{x}(\tau)=f(x(\tau)) \neq 0$ is bounded for $\tau\ge0$,
    and so are the derivatives $\ddot{x}=Df(x)\dot{x}$ and $\dddot{x}=D^2f(x)(\dot{x},\dot{x})
     + Df(x) \ddot{x}$. Moreover, by a Gronwall estimate on an interval of length $1$ we find a constant $C>0$ such that
     \begin{equation*}
       \| \ddot{x}(t)\| + \|\dddot{x}(t)\| \le C \|\dot{x}(\tau)\| \quad \text{for all} \quad \tau,t\ge 0 \text{ with } |\tau-t|\le 1.
     \end{equation*}
          Therefore, the following Taylor expansions hold with remainders that are uniform for
     $0<h \le 1$ and  $0<h \le \tau < \infty$:
     \begin{equation} \label{eq5:expand}
       \begin{aligned}
         \dot{x}(\tau-h)& = \dot{x}(\tau) - h \ddot{x}(\tau) + \cO(h^2\|\dot{x}(\tau)\|),\\
         x(\tau+h)& = x(\tau)+ h \dot{x}(\tau) + \frac{1}{2}h^2 \ddot{x}(\tau) + \cO(h^3\|\dot{x}(\tau)\|),\\
         x(\tau-h)& =x(\tau)-h \dot{x}(\tau) + \frac{1}{2}h^2 \ddot{x}(\tau) + \cO(h^3\|\dot{x}(\tau)\|).
       \end{aligned}
     \end{equation}
     Note that we will not assume any lower bound for $\|\dot{x}(\tau)\|$.
        Using \eqref{eq5:expand}, we expand the first term on the righthand side as in 
    \eqref{eq5:expandnextV}:
    \begin{equation*}
      \begin{aligned}
        \|\dot{x}(\tau-h)\|^2& = \|\dot{x}(\tau)\|^2 - 2 h \ddot{x}(\tau)^{\top}\dot{x}(\tau) +
        \cO(h^2\|\dot{x}(\tau)\|^2),\\
        \|\dot{x}(\tau-h)\|^{-2}& = \|\dot{x}(\tau)\|^{-2} + 2 h \ddot{x}(\tau)^{\top}\dot{x}(\tau)
        \|\dot{x}(\tau)\|^{-4}+ \cO(h^2\|\dot{x}(\tau)\|^{-2}),\\
        P_{\dot{x}(\tau-h)}& =\|\dot{x}(\tau-h)\|^{-2} \dot{x}(\tau-h)\dot{x}(\tau-h)^{\top}= P_{\dot{x}(\tau)}\\
        & - h \|\dot{x}(\tau)\|^{-2} \Big[ \ddot{x}(\tau)\dot{x}(\tau)^{\top}+ \dot{x}(\tau) \ddot{x}(\tau)^{\top} 
         -2 \ddot{x}(\tau)^{\top}\dot{x}(\tau) P_{\dot{x}(\tau)} \Big] + \cO(h^2).
      \end{aligned}
      \end{equation*}
      In a similar fashion, we find from \eqref{eq5:expand}
            \begin{equation*}
        \begin{aligned}
          \|x(\tau \pm h)-x(\tau)\|^{-2}& =h^{-2} \|\dot{x}(\tau)\|^{-2}\left(1\mp h \|\dot{x}(\tau)\|^{-2}
          \ddot{x}(\tau)^{\top}\dot{x}(\tau) +\cO(h^2) \right), \\
          P_{x(\tau \pm h)-x(\tau)}& = P_{\dot{x}(\tau)}\pm h \|\dot{x}(\tau)\|^{-2} \Big[ \frac{1}{2}
            (\ddot{x}(\tau)\dot{x}(\tau)^{\top}+ \dot{x}(\tau)\ddot{x}(\tau)^{\top}) \Big. \\
            & \Big. \qquad \qquad \qquad \qquad \quad - \ddot{x}(\tau)^{\top}
            \dot{x}(\tau) P_{\dot{x}(\tau)} \Big] +\cO(h^2).
          \end{aligned}
            \end{equation*}
            Inserting both expansions into  \eqref{eq5:sinest}, we obtain that the righthand side is
            $\cO(h^2)$.  Finally, using the Lipschitz bound for $\arcsin$ near zero, we find  
            that the difference of the angles is $\cO(h^2)$, as claimed in \eqref{eq5:esttau}.
   \end{proof}

   As a consequence of Proposition \ref{prop5:estdiff} we obtain an estimate for the approximate
   values \eqref{eq5:alpha1V0} of the angular spectrum.

   \begin{theorem} \label{thm5:unifalpha1}
     Let the assumptions of Proposition \ref{prop5:estdiff} hold. Then there exist constants
     $C,h_0>0$ such that the values $\alpha_1^{\mathrm{diff}}(h,T,V_0)$ with initial subspace
     $V_0=\mathrm{span}(\dot{x}(0))$ satisfy
     \begin{equation} \label{eq5:unifalpha1}
      | \alpha_1^{\mathrm{cont}}(T,V_0)- \alpha_1^{\mathrm{diff}}(h,T,V_0)| \le Ch
     \end{equation}
     for all $T>0$ and $h=\frac{T}{N} \le h_0$.
   \end{theorem}
   \begin{proof} By Theorem \ref{th3:unifapproxT} it is enough to prove the estimate \eqref{eq5:unifalpha1}
     for $\alpha_1^{\mathrm{step}}$ rather than $\alpha_1^{\mathrm{cont}}$. Such an estimate follows
     immediately from \eqref{eq5:vareqtau} and Proposition \ref{prop5:estdiff}
        \begin{equation*}
       \begin{aligned}
         &  |\alpha_1^{\mathrm{diff}}(h,T,V_0)- \alpha_1^{\mathrm{step}}(h,T,V_0)|  \\
           & = \frac{1}{hN} \Big| \sum_{j=1}^N \ang(x((j-1)h)-x(jh),x(jh)-x((j+1)h)) - \ang(\dot{x}((j-1)h),
         \dot{x}(jh)) \Big|\\
         & \le \frac{1}{hN} \sum_{j=1}^N C h^2 = Ch.
       \end{aligned}
     \end{equation*}
   This finishes the proof.  
          \end{proof}

   As in \eqref{eq4:limindiv} let us introduce the angular range generated by the subspace
   $V_0=\mathrm{span}(\dot{x}(0))$  for a  fixed stepsize $h>0$:
 \begin{equation*} \label{eq4:limdiff}
         I_1^{\mathrm{diff}}(h,V_0) =
         \left[\varliminf_{N\to \infty} \alpha_1^{\mathrm{diff}}(h,Nh,V_0) ,
           \varlimsup_{N\to \infty}\alpha_1^{\mathrm{diff}}(h,Nh,V_0) \right].
  \end{equation*}
 We conclude with the convergence of
 the angular ranges with respect to the Hausdorff metric. The proof uses the same arguments
 as in Theorem \ref{th4:approxspechaus}, except  that we cannot proceed from
 the angular range to the angular spectrum since we consider only  the
 subspace spanned by the flow direction.

 \begin{theorem} \label{th5:hausrangediff}
   Let the assumptions of Proposition \ref{prop5:estdiff} hold. Then there exist constants $C,h_0>0$ such that
     \begin{equation*} \label{eq4:specHausdorff}
              d_H(I_1^{\mathrm{diff}}(h, V_0),I_1^{\cont}(V_0))\le C h, \quad \forall \, 0<h \le h_0
     \end{equation*}
    holds for  the Hausdorff distance $d_H$.
 \end{theorem}

\section*{Use of AI tools declaration}
During the preparation of this work, the authors used ChatGPT (OpenAI)
to assist with language editing, the critical examination of
mathematical arguments, and the development of MATLAB code for
numerical experiments. The authors reviewed and verified the resulting
material and take full responsibility for the content of this paper.

\section*{Acknowledgments}
Both authors are grateful to the Research Centre for Mathematical
Modelling ($\text{RCM}^2$) at Bielefeld University for continuous
support of their joint research. 
The work of WJB was funded by the Deutsche Forschungsgemeinschaft
(DFG, German Research Foundation) – SFB 1283/2 2021 – 317210226. 
TH thanks the Faculty of Mathematics at Bielefeld
University for further support.
 

\bibliographystyle{abbrv}

\end{document}